\documentclass[journal]{IEEEtran}
\usepackage{graphicx} % Required for inserting images
\usepackage{tikz}
\usepackage{tikzpeople}
\usetikzlibrary{shapes.symbols}
\usetikzlibrary{shapes.geometric}
\usetikzlibrary{positioning,fit,calc}
\usetikzlibrary{backgrounds}
\usepackage{circuitikz}
\usetikzlibrary{arrows}
\usepackage{amsmath, amssymb, nccmath}
\usepackage{physics}
\usepackage{url}
\usepackage{cuted}
\usepackage{booktabs}
\usepackage{cite}

\usepackage{siunitx} % Required for alignment
\DeclareMathOperator*{\diag}{\text{diag}}

\newcommand{\iu}{{\mathrm{\mathbf{i}}\mkern1mu}}

\newcommand{\scada}{\text{SCADA}}
\newcommand{\pmu}{\text{PMU}}

\usepackage{amsthm}
\newtheorem{definition}{Definition}
\newtheorem{theorem}{Theorem}
\newtheorem{lemma}[theorem]{Lemma}

\title{Power System Steady-State Estimation Revisited}
\author{Pavel Rytir, Ales Wodecki, Martin Malachov, Pavel Baxant, Premysl Vorac, Miloslava Chladova, Jakub Marecek}

\date{\today}

\begin{document}

\maketitle

\begin{abstract}
In power system steady-state estimation (PSSE), one needs
to consider
(1) the need for robust statistics, (2) the nonconvex transmission constraints,
(3) the fast-varying nature of the inputs, and the corresponding need to track optimal trajectories as closely as possible. In combination, these challenges have not been considered, yet. %despite the nature of the problem that calls for such tools.
In this paper, we address all three challenges. The need for robustness (1) is addressed by using an approach based on the so-called Huber model. The non-convexity (2) of the problem, which results in first order methods failing to find global minima, is dealt with by applying global methods. One of these methods is based on a mixed integer quadratic formulation, which provides results of several orders of magnitude better than conventional gradient descent. Lastly, the trajectory tracking (3) is discussed by showing under which conditions the trajectory tracking of the SDP relaxations has meaning.
%We illustrate the results on  a subset of the Czech transmission system.
\end{abstract}

\section{Introduction}

Modern transmission systems are run by increasingly sophisticated energy management systems (EMS). 
Their functionality relies on the availability of an estimate of the state of the system.
The estimates also have an indispensable role in the situational awareness of human operators in the control room \cite{prostejovsky2019future}.
%(See Figure \ref{fig:control-center} in the Appendix.)
%  With the coming deployment of machine learning \cite{marot2021learning,biagini2020autonomous},
There is a long history of work on state estimation in power systems, dating back to \cite[e.g.]{schweppe1970powerA,schweppe1970powerB,schweppe1970powerC}, % ,schweppe1974static,monticelli2012state
surveyed in numerous excellent reviews \cite[e.g.]{monticelli2000electric,abur2004power,huang2012state,kekatos2017psse,10007691,Khalili2022,cheng2023survey}. % darmis2023survey,

Here, we highlight three challenges that have not been considered in much detail previously. 
First, we stress the need for robust statistics. 
Considering the extensive instrumentation of modern power systems with sensors, one can allow for a certain fixed number of measurements to be ignored, but what measurements to ignore is a non-trivial question. 
The number of measurements that can be ignored depends on an NP-hard subproblem bounding the number of roots of a multi-homogeneous systems and is complicated further by the common practice of having multiple sensors at a single bus.  

Second, we illustrate the non-convex nature of the state-estimation problem, even without the robustness considerations. Indeed, one should utilize the transmission constraints in the alternating-current model, which can be modelled as non-convex polynomials \cite{Ghaddar14,molzahn2017survey}. Commonly considered linearizations may suffer substantial errors in practice \cite{cain2012history}. Building on the work of \cite{8855013,9030243}, we show that these errors can be arbitrarily large even in very simple, 2-bus systems.

Finally, we stress the need to track the trajectory of optimal solutions as closely as possible, with as little delay as possible. Indeed, even if one were able to solve the non-convex optimization problem providing robust estimator exactly, if this takes too long, it may not be particularly helpful. 

We present an approach that combines:
\begin{itemize}
\item robust statistics in the Huber model, implemented using mixed-integer programming,
\item robust solvers for the non-convex optimization problems, utilizing certain convexifications,
\item novel approaches to tracking a trajectory of optimal solutions to the convexifications.
\end{itemize}
%In all three aspects, our approach is novel. 

%\begin{itemize}
%\item covering modern sensors and errors in the respective measurement chains
%\item formulating most-likely estimates and analysing their statistical properties 
%\item explaining clearly the distinction between convex and non-convex optimization problems involved
%\item detailing what features have been developed and utilized by what TSOs in Europe and beyond. 
%\end{itemize}

%\textcolor{red}{
%Things to add to the intro (Jakub will draft?):
%\begin{itemize}
%    \item Need for robust statistics,
%    \item Non-convex nature of the state optimization problem,
%    \item Time-varying nature,
%    \item add a comment on the complexity of the cost function in terms of the various physical quantities - discussion,
%\end{itemize}
%}

\subsection{Notation}

\begin{tabular}{lp{6.2cm}}
    $\mathcal{N} \subset \mathbb{N}$ & The set of indices corresponding to the buses of the network \\
    $\mathcal{M} \subset \mathbb{N}$ & The set of indices corresponding to the branches of the network \\    
$\mathcal{G}=\left(\mathcal{N},\mathcal{M}\right)$ & The graph that describes the network \\
    $N$ & The number of measurements available \\   
    $v_k$ &  True complex voltage at bus $k \in \mathcal{N}$ \\
    $\theta_k$ & True phase angle of complex voltage at bus $k \in \mathcal{N}$ with respect to a fixed reference bus\\
    $v_k^\pmu$ & PMU measurement of complex voltage at bus $k \in \mathcal{N}$ \\
    $\theta_k^\pmu$ & PMU measurement of voltage phase angle at bus $k \in \mathcal{N}$ with respect to the fixed reference bus \\
    $\lvert v_k\rvert$ &  True voltage amplitude at bus $k \in \mathcal{N}$ \\
    $\lvert v_k\rvert^{\scada}$ &  SCADA voltage amplitude measurement at bus $k \in \mathcal{N}$ \\
    $p_{f,l}$ & True active power flow at \textit{from} end of branch $l \in \mathcal{M}$ \\
    $p_{t,l}$ & True active power flow at \textit{to} end of branch $l \in \mathcal{M}$ \\
    $q_{f,l}$ & True reactive power flow at \textit{from} end of branch $l \in \mathcal{M}$ \\
    $q_{t,l}$ & True reactive power flow at \textit{to} end of branch $l \in \mathcal{M}$ \\
    $p^\scada_{f,l}$ & SCADA active power measurement at \textit{from} end of  branch $l \in \mathcal{M}$\\
    $p^\scada_{t,l}$ & SCADA active power measurement at \textit{to} end branch $l \in \mathcal{M}$\\
    $q^\scada_{f,l}$ & SCADA reactive power measurement at \textit{from} end of  branch $l \in \mathcal{M}$\\
    $q^\scada_{t,l}$ & SCADA reactive power measurement at \textit{to} end branch $l \in \mathcal{M}$\\
    $p_{k}$ & True injected active power at bus $k \in \mathcal{N}$ \\
    $q_{k}$ & True injected reactive power at bus $k \in \mathcal{N}$\\
    $p_{k}^{\scada}$ & SCADA measured injected active power at bus $k \in \mathcal{N}$ \\
    $q_{k}^{\scada}$ & SCADA measured injected reactive power at bus $k \in \mathcal{N}$\\
    $i_{f,l}$ & True complex current injection at \textit{from} end of the line $l \in \mathcal{M}$\\    
    $i_{t,l}$ & True complex current injection at \textit{to} end of the line $l \in \mathcal{M}$
\end{tabular}

    \begin{tabular}{lp{6.2cm}}
     $\theta^i_{f,l}$ & True phase angle of the current injection at \textit{from} end of the branch $l \in \mathcal{M}$ with respect to the voltage at the reference bus\\
    $\theta^i_{t,l}$ & True phase angle of the current injection at the \textit{ to} end of the branch $l \in \mathcal{M}$ with respect to the voltage at the reference bus\\   
    $i_{f,l}^{\pmu}$ & PMU measured complex current at \textit{from} end of line $l \in \mathcal{M}$ with respect to the reference bus\\

    $i_{t,l}^{\pmu}$ & PMU measured complex current at \textit{to} end of line $l \in \mathcal{M}$ with respect to the reference bus\\ 
    $l_{f}, l_{t} \in \mathcal{N}$ & denote the source and destination buses for the line $l \in \mathcal{M}$, respectively \\
    $i^\text{loc}_{f,l}$ & True complex current at \textit{from} end of line $l \in \mathcal{M}$ with respect to bus $l_{f} \in \mathcal{N}$\\
    $i^\text{loc}_{t,l}$ & True complex current at \textit{to} end of line $l \in \mathcal{M}$ with respect to bus $l_{t} \in \mathcal{N}$\\
    $\left|i_{f,l}\right|$ & True current amplitude at \textit{from} end of line $l \in \mathcal{M}$\\
    $\left|i_{t,l}\right|$ & True current amplitude at \textit{to} end of line $l \in \mathcal{M}$\\
    $\left|i_{f,l}\right|^{\scada}$ & SCADA measured current amplitude at \textit{from} end of line $l \in \mathcal{M}$\\
    $\left|i_{t,l}\right|^{\scada}$ & SCADA measured current amplitude at \textit{to} end of line $l \in \mathcal{M}$\\
    $\varphi_{f,l}$ & True phase angle of the current at \textit{from} end of branch $l \in \mathcal{M}$ with respect to the voltage at bus $l_{f} \in \mathcal{N}$\\
    $\varphi_{t,l}$ & True phase angle of the current at \textit{to} end of branch $l \in \mathcal{M}$ with respect to the voltage at bus $l_{t} \in \mathcal{N}$\\

    $i_{k}$ & True complex current injection at bus $k \in \mathcal{M}$ with respect to the reference bus\\
    $i_{k}^{\pmu}$ & PMU measured complex current injection at bus $k \in \mathcal{N}$ with respect to the reference bus\\
    $\left|i_{k}\right|$ & True amplitude of current injection at bus $k \in \mathcal{N}$\\  
    $i_{k}^{\text{loc}}$ & True complex current injection at bus $k \in \mathcal{N}$ with respect to the voltage at bus $k \in \mathcal{N}$\\
    $\left|i_{k}\right|^{\scada}$ & SCADA measured amplitude of current injection at bus $k \in \mathcal{N}$\\
%\end{tabular}
%   \begin{tabular}{lp{6.2cm}}    
    $\varphi_{k}$ & True phase angle of the current injected at bus $k \in \mathcal{N}$ with respect to the voltage at bus $k$\\
    $\varphi^\scada_{k}$ & phase angle of the current injected at bus $k \in \mathcal{N}$ with respect to the voltage at bus $k$ calculated from SCADA measurements of injected active and reactive power \\
    
\end{tabular}

\section{Sensors and the Measurement Chain}\label{sec_sensors_and_meas}
Any method of phase estimation and subsequent management actions, stemming from a chosen method of optimization and algorithmic implementation, rely on sensory data obtained by querying available sensors. The quantities of interest are the following:
\begin{itemize}
    \item voltage magnitude at buses,
    \item synchronized phase of voltage at buses,
    \item magnitude of current flows at the both ends of a power line, which need not be the same due to losses, 
    \item synchronized phase of the current at both ends of a power line.
\end{itemize}
The solution of the power system state estimation problem (PSSE) gives the complex-valued voltage and current
\cite{steinmetz1893complex}, also known as phasors, for the entire network, given some of the above. The related key challenges include:
\begin{itemize}
    \item Some of the quantities are not available by direct sensory measurement. %In this case, quantities such as energy consumption (load at tap points) may be used to provide estimates for some of the quantities. 
    %Due to the possible lack of synchronicity of the data gathering process, it may happen that two "flows" in the opposite direction will appear as a zero, which is not representative of the true situation.
    \item The measurements of voltage and current can be obtained by a sensory measurement, but as with any measurement, this is subject to imprecisions stemming from the accuracy of the device measurements or the asynchronous nature in which the measurements may happen.
\end{itemize}
%For these reasons, the placement, accuracy and synchronicity of the measurement devices is a chief concern, which influences any further data processing or algorithmic implementations that may be applied to the gathered data. 

The state-of-the-art sensors available to get the full phasor data are synchronized phasor measurement units (PMU, \cite{phadke2002synchronized}). 
PMU units are synchronized using the GPS universal clock, providing ultimate synchronicity \cite{phadke2002synchronized}. 
To achieve the high precision that PMUs are known for, yoke or differential sensors are used to concentrate the field lines before measurement. The core of the sensor is made out of nonlinear magnetic materials, which exhibit magnetisation variations due to the current induced magnetic field. It has been documented that PMU's are the most precise of the instruments available on the market today \cite{bale20, prostejovsky2019future}.

%OPTIONS FOR FURTHER DISCUSSION:
%\begin{itemize}
%    \item other sensors: hall sensors are not bad, they are less precise, but hcan measure higher voltages. Not sure whether this comes into play, would need to have a look at some of the more technical papers,
%    \item if fully synchronized sensors are bought, they need to be put to work together with the current sensory setup. This is a non-trivial task, we may describe this.
%\end{itemize}

%\begin{figure}
%    \centering
%    \input{diagram2.tex}
%    \caption{State estimation}
%    \label{fig:state-estimation}
%\end{figure}

In robust statistics, one assumes that a random process modeling the sensory readout mostly produces measurements that follow a certain structure; these measurements are called non-corrupt. However, some measurement outcomes may involve potentially unbounded errors; these are referred to as corrupted measurements. In the following, we denote $K$ as the set of measurement indices that correspond to non-corrupted measurements.

\subsection{SCADA Mesurements}
We describe measurement chains in similar way as in~\cite{Cheng2023ModelingOS}.

\begin{figure}[h]
    \centering
    \begin{tikzpicture}[scale=0.8, every node/.style={scale=0.8}]
    \node (voltage_true)  {$\lvert v_k\rvert$};
    \node (voltage1) [right=0.9 cm of voltage_true] {$\lvert v_k\rvert'$};
    \node (voltage2) [right=0.9 cm of voltage1] {$\lvert v_k\rvert''$};
    \node (voltage3) [right=1.7 cm of voltage2] {$\lvert v_k\rvert'''$};
    \node (voltage4) [right=0.9 cm of voltage3] {$\lvert v_k\rvert^\scada$};

    \draw [->] (voltage_true) -- node[above, align=center] {$e^V_{VT}$}  (voltage1);
    \draw [->] (voltage1) -- node[above, align=center] {$e^V_{CAB}$} (voltage2);
    \draw [->] (voltage2) -- node[above, align=center] {$e^V_{IED}$} (voltage3);
    \draw [->] (voltage3) -- node[above, align=center] {$e^V_{CN}$} (voltage4);
    
    \node (voltage_angle_true) [below=0.2 cm of voltage_true] {$\theta_k$};
    \node (voltage_angle1) [below=0.2 cm of voltage1] {$\theta_k'$};
    \node (voltage_angle2) [below=0.2 cm of voltage2] {$\theta_k''$};
    \draw [->] (voltage_angle_true) -- node[above, align=center] {$e^\theta_{VT}$}  (voltage_angle1);
    \draw [->] (voltage_angle1) -- node[above, align=center] {$e^\theta_{CAB}$}  (voltage_angle2);

    \node (current_true) [below=0.2 cm of voltage_angle_true] {$\lvert i_{f,l}\rvert$};
    \node (current1) [below=0.2 cm of voltage_angle1] {$\lvert i_{f,l}\rvert'$};
    \node (current2) [below=0.2 cm of voltage_angle2] {$\lvert i_{f,l}\rvert''$};
    \draw [->] (current_true) -- node[above, align=center] {$e^I_{CT}$}  (current1);
\draw [->] (current1) -- node[above, align=center] {$e^I_{CAB}$}  (current2);

    \node (current_angle_true) [below=0.2 cm of current_true] {$\theta^I_{f,l}$};
    \node (current_angle1) [below=0.2 cm of current1] {$\theta^{I'}_{f,l}$};
    \node (current_angle2) [below=0.2 cm of current2] {$\theta^{I''}_{f,l}$};
    \draw [->] (current_angle_true) -- node[above, align=center] {$e^{\theta^I}_{CT}$}  (current_angle1);
    \draw [->] (current_angle1) -- node[above, align=center] {$e^{\theta^I}_{CAB}$}  (current_angle2);

    \node (power_active3) [below=0.5 cm of voltage3] {$p_{f,l}'''$};
    \node (power_active4) [below=0.5 cm of voltage4] {$p_{f,l}^\scada$};
    \draw [->] (power_active3) -- node[above, align=center] {$e^P_{CN}$}  (power_active4);

    \node (power_active2) [left=0.3 cm of power_active3] {$p_{f,l}''$};
     \draw [->] (power_active2) -- node[above, align=center] {$e^P_{IED}$}  (power_active3);
    
    \node (power_reactive3) [below=0.5 cm of power_active3] {$q_{f,l}'''$};
    \node (power_reactive4) [below=0.5 cm of power_active4] {$q_{f,l}^\scada$};
 \draw [->] (power_reactive3) -- node[above, align=center] {$e^Q_{CN}$}  (power_reactive4);

 \node (power_reactive2) [left=0.3 cm of power_reactive3] {$q_{f,l}''$};
 \draw [->] (power_reactive2) -- node[above, align=center] {$e^Q_{IED}$}  (power_reactive3);

    \begin{scope}[on background layer]
    \node (pq) [box=gray!10, draw=black!50, fit={(power_reactive2) (power_active2)}] {};
    \end{scope}

    \draw [->] (voltage2) -- (pq);
    \draw [->] (voltage_angle2) -- (pq);
    \draw [->] (current2) -- (pq);
    \draw [->] (current_angle2) -- (pq);

    \end{tikzpicture}
    \caption{SCADA measurement chain}
    \label{fig:scada_measurement_chain}
\end{figure}
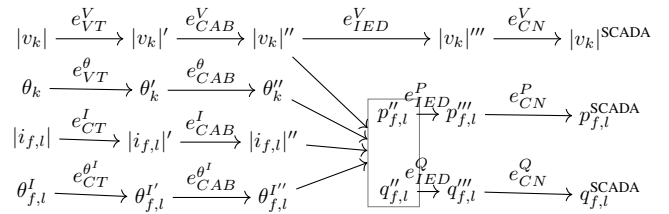

\subsubsection{SCADA Voltage amplitude measurement}

\begin{equation}
 \lvert v_k\rvert^\scada \approx \begin{cases} \lvert v_k\rvert(1+e^V_{VT})+e^V_{CAB}+ &   \\
  +e^V_{IED}+e^V_{CN} & \forall k \in K \\
 \textrm{ arbitrary } & \textrm{ otherwise }
 \end{cases}
\end{equation}
where $e^V_{VT}$ is an error introduced by the voltage transformer (VT). It consists of two types errors: systematic and random. Systematic errors may be introduced by non fully calibrated VT, as disscussed in the IEEE Std C57.13-2016. 
The absolute value of the sum of both parts in $e^V_{VT}$ is bounded from above by an upper bound depending on the class of the transformer. 
%We assumme that the random error follows a non-gaussian distribution which we model as a Gassian mixture model.

$e^V_{CAB}$ is an error introduced by control cables and burden resistors. We assume that this error follows a non-zero-mean Gaussian distribution.

$e^V_{IED}$ is an error introduced by an intelligent electronic device such as A/D and D/A converters. We will assume that these devices will introduce an error with zero-mean Gaussian distribution.

$e^V_{CN}$ is an error introduced by a communication network. When the results are transmitted over a network, data could get lost or delayed, so they will not arrive in time for a periodic stready-state estimation (SSE) computation. We assume that this error follows a Gaussian mixture model (GMM, \cite{mclachlan1988mixture,vsmidl2006variational}) distribution.

% We follow similar notation as in~\cite{dobakhshari2021robust}. Let $k$ be a bus. Let $\lvert V_k\rvert^\scada$ be an amplitude measurement. Then
% \begin{equation}
%     V_k=\lvert V_k\rvert^\scada e^{\iu \theta_k}+\varepsilon^{V_k} e^{\iu \theta_k}
% \end{equation}
% where $\varepsilon^{V_k}$ is the measurement error.

\subsubsection{SCADA Branch Power Measurements}
\label{sec:scada_branch_flow}

\begin{equation}
    \lvert v_k\rvert''=\lvert v_k\rvert(1+e^V_{VT})+e^V_{CAB}
\end{equation}
\begin{equation}
    \theta''_{f,l}=\theta_{f,l}+e^\theta_{VT}+e^\theta_{CAB}
\end{equation}
\begin{equation}
    \lvert i_{f,l}\rvert''=\lvert i_{f,l}\rvert(1+e^T_{CT})+e^I_{CAB}
\end{equation}
\begin{equation}
    \theta^{i''}_{f,l}=\theta^I_{f,l}+e^{\theta^I}_{CT}+e^{\theta^I}_{CAB}
\end{equation}

\begin{equation}
    p_{f,l}^\scada \approx \begin{cases}\lvert v_k\rvert''\lvert i_{f,l}\rvert''\cos{(\theta_k''-\theta^{I''}_{f,l})}+ & \\
    +e^P_{IED}+e_{CN}^P & \forall k \in K\\
    \textrm{ arbitrary } & \textrm{ otherwise }
    \end{cases}
\end{equation}

\begin{equation}
    q_{f,l}^\scada \approx \begin{cases}\lvert v_k\rvert''\lvert i_{f,l}\rvert''\sin{(\theta_k''-\theta^{i''}_{f,l})}+ & \\
    +e^Q_{IED}+e_{CN}^Q & \forall k \in K\\
    \textrm{ arbitrary } & \textrm{ otherwise }
    \end{cases}
\end{equation}
where the errors are analogous as in the SCADA voltage amplitude section.

% let $k,l$ be buses. We will consider complex current $I^\text{loc}_{kl}$ through a line between $k$ and $l$.

% \begin{equation}
%     I^{loc}_{kl}=(\lvert I_{kl}\rvert^\scada + \varepsilon^I_{kl})e^{i (\varphi^\scada_{kl}+\varepsilon^\varphi_{kl})}
% \end{equation}
% where $\varphi^\scada_{kl}=\arctan(-\frac{Q^\scada_{kl}}{P^\scada_{kl}})$
% and $\varepsilon^I_{kl}$ and $\varepsilon^\varphi_{kl}$ are measurement errors.

% Note that
% \begin{equation}
%     I_{kl}=I^{loc}_{kl}e^{\iu \theta_k}.
% \end{equation}

% Utilizing transmission line $\pi$ model, we have
% \begin{equation}
%     I^{loc}_{kl}e^{\iu \theta_k}=\left(\frac{y_{kl}}{2} + \frac{1}{z_{kl}}\right)V_k+\left(-\frac{1}{z_{kl}}\right)V_l
% \end{equation}

\subsubsection{SCADA Injection Measurements}
are computed analogously as in the previous section.

% The quantities $\lvert I_{\text{inj},k}\rvert^\scada$ and $\varphi^\scada_{\text{inj},k}$ could be measured in a similiar way as in \ref{sec:scada_branch_flow} or could be calculated using Ohm’s laws.

% \begin{equation}
%     I^\text{loc}_{\text{inj},k}=(\lvert I_{\text{inj},k}\rvert^\scada + \varepsilon^I_{k})e^{i (\varphi^\scada_{\text{inj},k}+\varepsilon^\varphi_{k})}
% \end{equation}
% where $\varepsilon^I_{k}$ and $\varepsilon^\varphi_{k}$ are measurement errors.

\subsection{PMU Measurements}
\begin{figure}[h]
    \centering
    \begin{tikzpicture}
    \node (voltage_true)  {$\lvert v_k\rvert$};
    \node (voltage1) [right=0.9 cm of voltage_true] {$\lvert v_k\rvert'$};
    \node (voltage2) [right=0.9 cm of voltage1] {$\lvert v_k\rvert''$};

    \draw [->] (voltage_true) -- node[above, align=center] {$e^V_{VT}$}  (voltage1);
    \draw [->] (voltage1) -- node[above, align=center] {$e^V_{CAB}$} (voltage2);

    \node (voltage_angle_true) [below=0.2 cm of voltage_true] {$\theta_k$};
    \node (voltage_angle1) [below=0.2 cm of voltage1] {$\theta_k'$};
    \node (voltage_angle2) [below=0.2 cm of voltage2] {$\theta_k''$};
    \draw [->] (voltage_angle_true) -- node[above, align=center] {$e^\theta_{VT}$}  (voltage_angle1);
    \draw [->] (voltage_angle1) -- node[above, align=center] {$e^\theta_{CAB}$}  (voltage_angle2);

    %\node (voltage_angle3) [below=0.2 cm of voltage3] {$\theta_k^\pmu$};

    \node (current_true) [below=0.2 cm of voltage_angle_true] {$\lvert i_{f,l}\rvert$};
    \node (current1) [below=0.2 cm of voltage_angle1] {$\lvert i_{f,l}\rvert'$};
    \node (current2) [below=0.2 cm of voltage_angle2] {$\lvert i_{f,l}\rvert''$};
    \draw [->] (current_true) -- node[above, align=center] {$e^I_{CT}$}  (current1);
\draw [->] (current1) -- node[above, align=center] {$e^I_{CAB}$}  (current2);

    \node (current_angle_true) [below=0.2 cm of current_true] {$\theta^I_{f,l}$};
    \node (current_angle1) [below=0.2 cm of current1] {$\theta^{I'}_{f,l}$};
    \node (current_angle2) [below=0.2 cm of current2] {$\theta^{I''}_{f,l}$};
    %\node (current_angle3) [below=0.2 cm of current3] {$\theta^{I,\pmu}_{f,l}$};
    \draw [->] (current_angle_true) -- node[above, align=center] {$e^{\theta^I}_{CT}$}  (current_angle1);
    \draw [->] (current_angle1) -- node[above, align=center] {$e^{\theta^I}_{CAB}$}  (current_angle2);

    \node (voltage_phasor) [draw, right=0.4 cm of voltage2, align=center, yshift=-0.5cm] {Phasor\\Estimation\\(DFT)};
    \node (voltage3) [right=0.5 cm of voltage_phasor] {$v_k^\pmu$};

    \node (current_phasor) [draw, right=0.4 cm of current2, align=center, yshift=-0.5cm] {Phasor\\Estimation\\(DFT)};
    \node (current3) [right=0.5 cm of current_phasor] {$i_{f,l}^\pmu$};

    \draw[->] (voltage2) -- (voltage_phasor);
    \draw[->] (voltage_angle2) -- (voltage_phasor);

    \draw[->] (current2) -- (current_phasor);
    \draw[->] (current_angle2) -- (current_phasor);

    \draw[->] (voltage_phasor) -- (voltage3);
    %\draw[->] (voltage_phasor) -- (voltage_angle3);

    \draw[->] (current_phasor) -- (current3);
    %\draw[->] (current_phasor) -- (current_angle3);

    \end{tikzpicture}
    \caption{PMU measurement chain}
    \label{fig:scada_measurement_chain}
\end{figure}
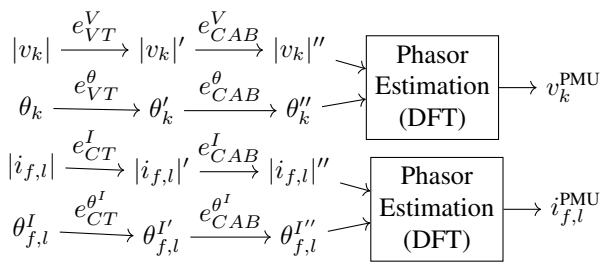

\subsubsection{PMU Voltage Phasor Measurements and Phase Angle Measurement}

\begin{equation}
 \lvert v_k^\pmu\rvert \approx \begin{cases} \lvert v_k\rvert(1+e^V_{VT})+e^V_{CAB} & \; \forall k \in K \\
 \textrm{ arbitrary } & \textrm{ otherwise }
 \end{cases}
\end{equation}

\begin{equation}
 \arg(v_k^\pmu)=\theta_k^\pmu \approx \begin{cases} \theta_k+e^\theta_{VT}+e^\theta_{CAB} & \; \forall k \in K \\
 \textrm{ arbitrary } & \textrm{ otherwise }
 \end{cases}
\end{equation}

where $e^V_{VT}$ and $e^V_{CAB}$ are the same errors as in the SCADA section.

Errors introduced by PMU unit: Sampling time errors, which are typically introduced by the frequency drift of oscilators;
off-nominal frequency errors due to changes in generation and load imbalances, 
GPS signal lost errors.
These errors lead to time-varying non-Gaussian distribution.

% Phasor measurement at bus $k$ equals:
% \begin{equation}
%     V_k=V_k^\pmu+\varepsilon^V_k
% \end{equation}
% where $\varepsilon^V_k$ is the measurement error.
% % \begin{equation}
% %     e^{\iu \theta_k}(\epsilon^V_k+\iu \epsilon^\theta_kV_k)
% % \end{equation}

% \subsubsection{PMU Phase-Angle Measurements}
% PMU can measure time-synchronized bus voltage phase angle, bus voltage amplitude and branches currents phasors. All measured angles are with respect to the phase angle of voltage at the reference bus.
% We express phase-angle measurement of voltage by first order Tailor approximation of $e^{\iu \theta_k}$. 
% \begin{equation}
%     e^{\iu \theta_k}\approx e^{\iu \theta^\pmu_k}+\iu e^{\iu \theta^\pmu_k}\varepsilon^\theta_k
% \end{equation}
% where $\varepsilon^\theta_k$ is a measurement error.

\subsubsection{PMU Current Phasor Measurements}

PMU can measure directly $I_{kl}$

\begin{equation}
 \lvert i_{f,l}^\pmu\rvert \approx \begin{cases} \lvert i_{f,l}\rvert(1+e^I_{CT})+e^I_{CAB} & \; \forall k \in K \\
 \textrm{ arbitrary } & \textrm{ otherwise }
 \end{cases}
\end{equation}

\begin{equation}
 \arg(i_{f,l}^\pmu)=\theta^{I,\pmu}_{f,l} \approx \begin{cases} \theta^I_{f,l}+e^{\theta^I}_{CT}+e^{\theta^I}_{CAB} & \; \forall k \in K \\
 \textrm{ arbitrary } & \textrm{ otherwise }
 \end{cases},
\end{equation}
where $e_{CT}^{I}$ denotes the current transformer error.

% \begin{equation}
%     I_{kl} = I^\pmu_{kl} + \varepsilon^I_{kl}
% \end{equation}
% where $\varepsilon^I_{kl}$ is the measurement error.

% %  e^{\iu \theta_{kl}}(\epsilon^I_{kl}+\iu \epsilon^\theta_{kl}I_{kl})

%\subsection{Covariances of Measurements Errors}
%TODO: ???

\section{Model}

\subsection{Power System State}\label{sec_power_voltage}
%We introduce Polar Power-Voltage Formulation bellow.

%Following \cite{zhang17}, one may describe the network by $\mathcal{G}=\left(\mathcal{N},\mathcal{M}\right),$ where $\mathcal{N},\mathcal{M}\subset\mathbb{N}$ are the sets of buses and lines, respectively. 
%In the following we assume that the network is described by a directed graph in order to give the from and two admittance matrices a well defined meaning.% Let $p$ and $q$ represent active and reactive power, respectively.

%Denote the complex voltage magnitudes and phase shifts at bus $k \in \mathcal{N}$ with respect to a reference bus by $\left|v_{k}\right|$ and $\theta_{k}$, respectively.

%The state of an $n$-bus system is a vector $x\in\mathbb{R}^{2n-1}$ of length $2n-1$.
%$$x=(\theta_2,\theta_3,\dots, \theta_n, \lvert v_1\rvert, \lvert v_2\rvert,\dots, \lvert v_n\rvert)^T$$
%where $\lvert v_i\rvert$ is the magnitude of voltage and $\theta_i$ is the phase angle at bus $i$. We assume that $\theta_1$ is a known angle at the reference bus and it is set to $0$. 

%An alternative to the above description detailed is a formulation of the power system dynamics depending on the rectangular representation of complex numbers \cite{zhang2017conic}. This formulation is advantageous because it leads to a polynomial optimization problem, for which guarantees of global convergence have been proven. More succinctly, one can show that there exists a hierarchy of semi-definite programs that converge to the global solution of the polynomial optimization problem.

The state of an $n$-bus system in the rectangular representation is a vector $v\in\mathbb{C}^{n}$ of length $n$.
$$v=(v_1, v_2,\dots, v_n)^T$$
We assume that the real component of voltage $v_1$ at the reference bus is non-negative.

\subsection{Branch Model}
We use the pi-model of a branch~\cite{matpower}. See Figure~\ref{fig:branch}.
\begin{figure}[h]
    \centering
    \begin{circuitikz}[american, scale=0.8]
    \ctikzset{t/.style={
        transformer core,
        inductors/width=1.8, 
        inductors/coils=8, 
        quadpoles/transformer core/height=2.4}
    }           
    \draw (0,4) node[t, anchor=A1](T){};
    \draw ($(T.A2)!0.5!(T.B2)$) node{$N:1$};
    \path (T.A1) -- node[] (Vf) {$v_f$} (T.A2);
    \draw[->,>=stealth',shorten >=5pt]
    (Vf) edge[] (T.A1);
    \draw[shorten >=4pt] (Vf) -- (T.A2);

    \path (T.B1) -- node[] (Vfn) {$\frac{v_f}{N}$} (T.B2);
    \draw[->,>=stealth',shorten >=4pt]
    (Vfn) edge[] (T.B1);
    \draw[shorten >=4pt] (Vfn) -- (T.B2);

    \draw[->,>=stealth'] (T.A1)+(0, 0.4) -- node[above] {$i_f$} ++(0.7, 0.4);
    \draw[->,>=stealth'] (T.B1)+(-0.7, 0.4) -- node[above] {$i_fN^*$} ++(0, 0.4);

    \node [ocirc] at (T.A1){};
    \node [ocirc] at (T.A2){};
    \draw (T.B1) to[short] ++(1,0) coordinate (1);
    \draw (T.B2) to[short] ++(1,0) coordinate (2);

    \draw (1) to[generic=$\frac{y^G_{ij}}{2}$] (2);

    \draw (1) to[generic=$y_{ij}$] ++(4,0) coordinate (3);
    \draw (2) to[short] ++(4,0) coordinate (4);

    \draw (3) to[generic=$\frac{y^G_{ij}}{2}$] (4);  

    \draw (3) to[short,-o] ++(1.5,0) coordinate (a);
    \draw (4) to[short,-o] ++(1.5,0) coordinate (b);

    \path (a) -- node[] (Vt) {$v_t$} (b);
    \draw[->,>=stealth',shorten >=5pt]
    (Vt) edge[] (a);
    \draw[shorten >=4pt] (Vt) -- (b);

    \draw[<-,>=stealth'] (a)+(-0.7, 0.4) -- node[above] {$i_t$} ++(0, 0.4);

\end{circuitikz}
    \caption{Transmission branch model (line or transformer). $N=\tau e^{\iu\theta_\textrm{shift}}$, where $\tau$ is a tap ratio and $\theta_\textrm{shift}$ is a phase shift of the transformer. If the branch is a line, then $N=1$. \textit{From} side of the branch is on the left, \textit{to} side is on the right.}
    \label{fig:branch}
\end{figure}
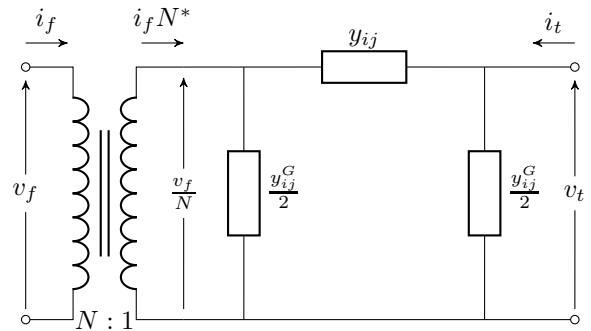

\subsection{Branch and Bus Admittance matrices}
%The branch admitance matrix $Y_{BR}$ is defined in the following way:

Let $1,\dots, n$ be buses.
The bus admittance $n\times n$ matrix $Y$ is defined in the following way:

\begin{equation*}
    Y_{ij}=
    \begin{cases}
    y^G_i+\sum_{\substack{k=1,\dots,n \\ k\neq i; k=(i,k)_t}}\left(\frac{y^G_{ik}}{2}+y_{ik}\right)\frac{1}{\tau^2}+\\
    +\sum_{\substack{k=1,\dots,n \\ k\neq i; k=(i,k)_f}}\left(\frac{y^G_{ik}}{2}+y_{ik}\right), \quad \text{if}\, i=j \\
    -y_{ij}\frac{1}{\tau e^{-\iu\theta_\textrm{shift}}},\quad \text{if}\, i\neq j, i = l_f, j=l_t \\
    -y_{ij}\frac{1}{\tau e^{\iu\theta_\textrm{shift}}},\quad \text{if}\, i\neq j, i = l_t, j=l_f 
    \end{cases}
\end{equation*}
where $y^G_i$ is admittance between bus $i$ and the ground and $y_{ij}$ is admittance between bus $i$ and bus $j$ (in case of parallel lines $y_{ij}$ is the sum of admittances of the parallel lines), and $y^G_{ij}$ is admittance between the line connecting bus $i$ and $j$ and the ground (in case of parallel lines, $y^G_{ij}$ is the sum of admittances between each parallel line and the ground), $\tau$ is a tap ratio of a transformer ($1$ if the branch is a line), $\theta_\textrm{shift}$ is a phase shift of a phase shifter ($0$ for a line).

Let $1,\dots,m$ be branches.
The "from" branch admittance $m\times n$ matrix $Y_f$ is defined:
\begin{equation}
    Y_{f,li}=
    \begin{cases}
        0, & \parbox[t]{4cm}{if branch $l$ is not connected to bus $i$} \\
        \left(y_{ij}+\frac{y^G_{ij}}{2}\right)\frac{1}{\tau^2}, & \parbox[t]{4cm}{if $i$ = $l_f$,  $j=l_t$}\\
        -y_{ij}\frac{1}{\tau e^{-\iu\theta_\textrm{shift}}}, & \parbox[t]{4cm}{if $i$ = $l_t$,  $j=l_f$}
    \end{cases}
\end{equation}

The "to" branch admittance $m\times n$ matrix $Y_t$ is defined:
\begin{equation}
    Y_{t,li}=
    \begin{cases}
        0, & \parbox[t]{4cm}{if branch $l$ is not connected to bus $i$} \\
        -y_{ij}\frac{1}{\tau e^{\iu\theta_\textrm{shift}}}, & \parbox[t]{4cm}{if $i$ = $l_f$,  $j=l_t$}\\
        y_{ij}+\frac{y^G_{ij}}{2}, & \parbox[t]{4cm}{if $i$ = $l_t$,  $j=l_f$}
    \end{cases}
\end{equation}

\subsection{Power Equations}\label{sec_rect_current_injection}
The power flows in the grid are governed by Kirchhoff's and Ohm's laws which can be expressed by the following matrix equations \cite{zhang2017conic}:
\begin{equation}
i=Yv,\quad i_{f}=Y_{f}v,\quad i_{t}=Y_{t}v,
\end{equation}
where $i=(i_1,\dots,i_n)^T$, $i_f=(i_{f,1},\dots,i_{f,m})^T$ and $i_t=(i_{t,1},\dots,i_{t,m})^T$.

The injected total power can be expressed by the following equation:

\begin{equation}
    s = p+q\iu = \diag\left(vv^{*}Y^{*}\right) = v\odot  (v^*Y^*)^T.
\end{equation}

where $\odot$ is the entry-wise (Hadamard) product, $p=(p_1,\dots,p_n)^T$ and $q=(q_1,\dots,q_n)^T$ are vectors of active and reactive power injections to the buses, respectively. Note that this equation is not linear.

%In this section we introduce the steady state problem formulation~\cite{cain2012history,molzahn2019survey,zohrizadeh2020conic}.

%https://www.ferc.gov/sites/default/files/2020-05/acopf-1-history-formulation-testing.pdf

\section{Steady State Estimation}
\subsection{State Estimation and Measurements}
In order to estimate $v$ we take a vector of measurements $z\in\mathbb{C}^L$ where $L=L^\scada + 2L^\pmu > 2n$, $L^\scada$ is the number of SCADA measurements and $L^\pmu$ is the number of PMU measurements, while simultaneously demanding a certain degree of uniformity of the measurements. In fact, one should keep in mind that the aforementioned method is just a heuristic and that the precise determination of the degrees of freedom is actually and NP-hard problem (see Section \ref{sec_robustness}). The measurements consist of voltage magnitudes at buses, injections to buses, active and reactive powers in network elements. SCADA measurements are real numbers. PMU measurements are complex. The measurement vector is related to the state vector by the following equation:
\begin{equation}
    z=h(v)+\eta,
\label{eq:statevsmeasurement}
\end{equation}
where $h:\mathbb{C}^n\mapsto\mathbb{C}^L$ is a non-linear map and $\eta\in\mathbb{C}^L$ is a 
%zero-mean Gaussian 
measurement noise vector.

The function $h(x)$ is derived from the power network admittance matrix and Ohm’s laws. See equations~(\ref{eq:measurements}--\ref{eq:measurements_last}).

% For a voltage measurement, the function $h(x)_i$ is $$z_i=V_k$$
% where measurement $i$ occurs at bus $k$.

% Let $y^G_{k}+\sum_{\substack{k=1,\dots,n \\ k\neq i}}y^G_{ik}/2=g_{k}+\iu b_{k}$ and $y_{kl}=g_{kl}+\iu b_{kl}$.
% For real power and reactive power measurements for a transmission line between buses $k$ and $l$, the function $h(x)_i$ is given by the following expressions:

% \begin{multline}
%     P_{kl}=V^2_kg_{kl}-\\
%     -V_kV_lg_{kl}\cos(\theta_k-\theta_l)-V_kV_lb_{kl}\sin(\theta_k-\theta_l)
% \label{eq:eq1}
% \end{multline}
% \begin{multline}
% Q_{kl}=-V^2_k(b_{kl}+b_k)-\\
% -V_kV_lg_{kl}\sin(\theta_k-\theta_l)+V_kV_lb_{kl}\cos(\theta_k-\theta_l)
% \label{eq:eq2}
% \end{multline}
% where the line between $k$ and $l$ has admittance $y_{kl}=g_{kl}+\iu b_{kl}$ ($g_{kl}>0$, $b_{kl}<0$ for inductive line), where $g_{kl}$ and $b_{kl}$ denote the line conductance and susceptance, respectively. Shunt susceptance and conductance at bus $k$ is denoted by $b_k$ and $g_l$, respectively. (If capacitive, then $b_k>0$). Since $b_k >> g_k$ usually in practice, we set $g_k$ to $0$ and leave it out. Note that the equations~\ref{eq:eq1} and \ref{eq:eq2} are non-linear, as they contain quadratic terms and sine and cosine functions.

% %TODO: Add also reformulation in standard basis?

% For each bus $p$ we have 
% \begin{equation}
%    0=-P_k+\sum_{l=1}^NP_{kl} 
% \end{equation}
% \begin{equation}
%     0=-Q_k+\sum_{l=1}^NQ_{kl}
% \end{equation}
% where $P_k$ and $Q_k$ are real and reactive power injections at bus $k$, respectively.

\subsection{Measurement Functions formulation}\label{sec_rectangular_voltage_formulation}
% \color{blue}
% Notation TODO: $v,p, q$ all need to be made lowercase so that notation like Yv (which is a matrix applied to a vector) does not read YV.

% relationships TO CHECK:
% \begin{itemize}
%     \item Admittance definition - at the end they were saying it's ok. But it may be simplified by removing some terms. Will have to be clarified with them
%     \item (17), should it be the component-wise conjugation instead of Hermitian conjugation? No. It was Y*v* instead v*Y*- should be fixed now
%     \item is there a missing complex unit missing in (25), (26)?
%     \item 
% \end{itemize}

% \color{black}

Let us define

\begin{align}\label{eq_measurements_bus}
E_{k}&=e_{k}e_{k}^{T}, \\
Y_{k,p}&=\frac{1}{2}\left(Y^{*}E_{k}+E_{k}Y\right), \\
Y_{k,q}&=\frac{\iu}{2}\left(E_{k}Y-Y^{*}E_{k}\right) \\
Y_{l,p_{f}}&=\frac{1}{2}\left(Y_{f}^{*}d_{l}e_{l_{f}}^{T}+e_{l_{f}}d_{l}^{T}Y_{f}\right), \\ Y_{l,p_{t}}&=\frac{1}{2}\left(Y_{t}^{*}d_{l}e_{l_{t}}^{T}+e_{l_{t}}d_{l}^{T}Y_{t}\right),\\
Y_{l,q_{f}}&=\frac{\iu}{2}\left(e_{l_{f}}d_{l}^{T}Y_{f}-Y_{f}^{*}d_{l}e_{l_{f}}^{T}\right), \\
Y_{l,q_{t}}&=\frac{\iu}{2}\left(e_{l_{t}}d_{l}^{T}Y_{t}-Y_{t}^{*}d_{l}e_{l_{t}}^{T}\right),
\label{eq_measurements_line_final}
\end{align}
where $e_{k}$ and $d_{l}$ are the $k$-th and $l$-th canonical basis vectors of $\mathbb{C}^{n}$ and $\mathbb{C}^{m}$, respectively.
Using the definitions above, one can compute voltages, active power and reactive power for the buses and lines using the following relationships
\begin{align}\label{eq:measurements}
\left|v_{k}\right|^{2}&=\Tr\left(E_{k}vv^{*}\right),\\ p_{k}&=\Tr\left(Y_{k,p}vv^{*}\right),  \\ q_{k}&=\Tr\left(Y_{k,q}vv^{*}\right), \\
p_{f,l}&=\Tr\left(Y_{l,p_{f}}vv^{*}\right), \\ q_{f,l}&=\Tr\left(Y_{l,q_{f}}vv^{*}\right), \\
p_{t,l}&=\Tr\left(Y_{l,p_{t}}vv^{*}\right), \\ q_{t,l}&=\Tr\left(Y_{l,q_{t}}vv^{*}\right). \label{eq:measurements_last}
\end{align}
Using the aforementioned relationships, which are quadratic in the complex voltage, allows us to formulate the problem of state estimation as a polynomial optimization problem.

\subsection{Polynomial Optimization Formulation}

Let us consider state estimation problem for a mixed-measurement setting involving SCADA and PMU measurements, in the so-called Huber  model \cite{huber2004robust} of robust statistics. Using this formalism, one assumes that some of the measurements are corrupted and should not be considered. This could also be seen as a joint problem of feature selection and power system state estimation.
This joint problem can be formulated either as a polynomial optimization problem (POP), or 
mixed-integer quadratically constrained quadratic program (MIQCQP).

In addition to the notation introduced in Sections \ref{sec_rectangular_voltage_formulation} and \ref{sec_rect_current_injection}, 
let $C$ with various sub-indices be the constant weights associated with SCADA and PMU measurements of different quantities in the objective function. Let $C_{v}^{\pmu}, C_{p,q}^{\scada}$ and $C_{v}^{\scada}$ be constants that balance the different physical interpretations of the quantities acquired by PMU and SCADA measurements, respectively. Making use of the relationships (\ref{eq_measurements_bus}--\ref{eq_measurements_line_final}), the problem may be formulated as

\begin{multline}
\label{eq_objective_of_poly_opt}
\underset{v,\beta,\gamma}{\min}\,\, C_{v}^{\pmu}\sum_{i\in I_{\text{\pmu}}}\beta^V_{i}\left|v_{i}-v_{i}^{\pmu}\right|^{2}+ \\
+C_{i}^{\pmu}\sum_{i\in I_{\text{\pmu}}}\beta^I_{i}\left|i_{i}-i_{i}^{\pmu}\right|^{2}+\\
+C_{i,tf}^{\pmu}\sum_{i\in I_{\pmu}}\beta^{I,f}_{i}\left|i_{f,i}-i_{f,i}^{\pmu}\right|^{2}+\beta^{I,t}_{i}\left|i_{t,i}-i_{t,i}^{\pmu}\right|^{2}+\\
+C_{p,q}^{\scada}\sum_{i\in I_{\scada}}\gamma_{i,1}\left(q_{i}-q_{i}^{\scada}\right)^{2}+\gamma_{i,2}\left(p_{i}-p_{i}^{\scada}\right)^{2}\\
+C_{v}^{\scada}\sum_{i\in I_{\scada}}\gamma_{i,3}\left|\left|v_{i}\right|^{2}-\left(\left|v_{i}\right|^{\scada}\right)^{2}\right| \\
+C_{p,q}^{\scada}\sum_{i\in I_{\scada}}\gamma_{i,4}\left(q_{i,f}-q_{i,f}^{\scada}\right)^{2}+\gamma_{i,5}\left(p_{i,f}-p_{i,f}^{\scada}\right)^{2}\\
+C_{p,q}^{\scada}\sum_{i\in I_{\scada}}\gamma_{i,6}\left(q_{i,t}-q_{k,t}^{\scada}\right)^{2}+\gamma_{i,7}\left(p_{i,t}-p_{i,t}^{\scada}\right)^{2}
\end{multline}
such that
\begin{align}\label{eq_pop_formulation_1} 
i&=Yv,\quad i_{f}=Y_{f}v,\quad i_{t}=Y_{t}v, \\ % Don't actually need this one
\left|v_{k}\right|^{2}&=\Tr\left(E_{k}vv^{*}\right), \quad p_{k}=\Tr\left(Y_{k,p}vv^{*}\right),\\ q_{k}&=\Tr\left(Y_{k,q}vv^{*}\right), \quad
p_{f,l}=\Tr\left(Y_{l,p_{f}}vv^{*}\right), \\
q_{f,l}&=\Tr\left(Y_{l,q_{f}}vv^{*}\right), \quad
p_{t,l}=\Tr\left(Y_{l,p_{t}}vv^{*}\right),  \\
q_{t,l}&=\Tr\left(Y_{l,q_{t}}vv^{*}\right) \\ 
d &= 
2\|\beta\|_{1}+\sum_{j}\|\gamma_{\cdot,j}\|_{1} \label{eq_robust_start}\\
(\beta^V_{i})^{2}&=\beta^V_{i},(\beta^I_{i})^{2}=\beta^I_{i},\label{eq_pop_formulation_pre_end}(\beta^{I,f}_{i})^{2}=\beta^{I,f}_{i},\\(\beta^{I,t}_{i})^{2}&=\beta^{I,t}_{i},
\gamma_{i,j}^{2}=\gamma_{i,j},
\label{eq_pop_formulation_0_end}
\end{align}

where $i, j$ in the last equality range through the appropriate indices, $d$ denotes the number of measurements considered, $i, i_{f}, i_{t}$ denote the injected currents at nodes and line from and to currents, respectively, and $\beta_{i}$ and $\gamma_{i,j}$ are measurement selection variables, which are ensured to be discrete by \eqref{eq_pop_formulation_pre_end} and \eqref{eq_pop_formulation_0_end}.
The norm is defined as follows:
\begin{equation}
\left\Vert v\right\Vert _{1}=\sum_{i=1}^{n}\left|v_{i}\right|.
\end{equation}

Due to the independence of the two sets of decision variables ($v$ versus $\beta$ and $\gamma$) one can view the problem above as a pessimistic independent bi-level problem \cite{wie2013, Chuong23}. In this setting, the inner problem becomes estimation of the state $v$ with a fixed set of $\beta$ and $\gamma$, the other problem is then a minimization with respect to the variables that select the measurements $\beta$ and $\gamma$. The optimal solution of this independent bi-level problem is trivially the solution of a single level optimization problem in which the state is fixed to the optimal state found by the bi-level. Assuming the additivity of the noise model (which is equivalent to assuming that the additive errors in the models presented in Section \ref{sec_sensors_and_meas} are all equal to 1) the solution of such a problem has been shown to exhibit statistical robustness similar to Lasso relaxations and is algorithmically tractable \cite{bertsimas2015best, huber2004robust}.

Through the decision variables $\beta_{i}$ and $\gamma_{i,j}$, the best subset selection method is incorporated in a way which mirrors the best state of the art results with proven convergence guarantees \cite{bertsimas2015best}. Another alternative way of incorporating best subset selection, would be the use of Lasso like relaxations, which would result in adding a term of the form 
\begin{equation}
r\left(2\|\beta\|_{1}+\sum_{j}\|\gamma_{\cdot,j}\|_{1}\right),    
\end{equation}
where $r > 0$ is the regularization and removing \eqref{eq_robust_start}, \eqref{eq_pop_formulation_pre_end} and \eqref{eq_pop_formulation_0_end}. from the constraints. This kind of formulation has proven guarantees for simpler (quadratic programming) problems in which mild assumptions on the noise and eigenvalue assumptions on the model matrix lead to desirable statistical properties \cite{green04, dono06, zhao06}.

In the polynomial optimization problem (\ref{eq_pop_formulation_1}--\ref{eq_pop_formulation_0_end}) one could specify a new variable 
\begin{equation}
w_{k}=\lvert v_{k}\rvert^{2},   
\end{equation}
which makes the cost function quadratic (instead of quartic) in the decision variables. Alternatively, one can define the variable
\begin{equation}\label{eq_additional_const_w_k}
\widetilde{w}_{k}^{2}=\left|v_{k}\right|^{2}   
\end{equation}
and modify the cost function by
\begin{align}
\begin{split}
\gamma_{i,3} \left(\left|v_{i}\right|^{2} - \left(\left|v_{i}\right|^{\scada}\right)^{2}\right)^{2} \\
\overset{\text{replaced by}}{\longrightarrow} \gamma_{i,3} \left(w_{k} - \left(\left|v_{i}\right|^{\scada}\right)^{2}\right)^{2}
\end{split}    
\end{align}
which leads to a least square interpretation of the cost function at the cost of an additional constraint \eqref{eq_additional_const_w_k}.

In order to treat quantities with the same physical significance consistently, we can also make the following exchange 
\begin{align}
    \begin{split}
\beta_{i}\left|v_{i}-v_{i}^{\pmu}\right|^{2}\overset{\text{replaced by}}{\longrightarrow}\beta_{i}\left(w_{k}-\left|v_{i}^{\pmu}\right|^{2}\right)^{2} \\+\widetilde{\beta}_{i}\left|\arg\left(v_{i}\right)-\arg\left(v_{i}^{\pmu}\right)\right|^{2},         
    \end{split}
\end{align}
where an algebraic replacement that is valid for a sufficient range of values would be used to replace $\arg$ in the formulation and implementation.

\subsection{Robustness of the estimate}\label{sec_robustness}
Due to the structure of the problem (\ref{eq_objective_of_poly_opt}--\ref{eq_pop_formulation_0_end}), the application of robust statistics captured by (\ref{eq_robust_start}--\ref{eq_pop_formulation_0_end}) and the variables $\beta_{k}, \gamma_{i,j}$ is not straightforward and requires some additional consideration. To this end, we describe an illustrative example which explains the fundamental difficulties of determining the degrees of freedom of the system, and follow it up by proposing a refinement of the robust part of (\ref{eq_objective_of_poly_opt}--\ref{eq_pop_formulation_0_end}). 

First, it is important to note that there exist two different views on degrees of freedom, which can be applied to measured systems. The first one of these is a fundamental concept that would recognize the power system in question as a system with $2n-1$ degrees of freedom, regardless of measurement outcomes. This approach however does not help us determine the number of measurements that need to be chosen in order to determine the state uniquely. To this end, a different concept of degrees of freedom is often introduced in which noiseless measurements are considered as physical constraints on the system that constrain the problem and change the degrees of freedom \cite{Marecek2014PowerFA, Baillieul82}. The term degrees of freedom will be henceforth used in this sense.

The correct choice of $d$ in (\ref{eq_robust_start}--\ref{eq_pop_formulation_0_end}) is non-trivial and the algebraic degrees of freedom of such systems are not known exactly. In theory, this problem has been studied and it has been shown that adding measurement data constraining the system changes the degrees of freedom in a non-trivial way \cite{Marecek2014PowerFA, Baillieul82} and the precise determination of the degrees of freedom is an NP-hard problem \cite{Marecek2014PowerFA} and thus any precise method that determines a suitable number of measurements would be costly. Section \ref{sec_robust_stat} details how these issues may be partially addressed numerically.

\section{A Numerical Study}\label{sec_num_study}

\begin{table*}
\centering
\caption{Estimating the states of IEEE test cases without noise
%using different stopping times $t_{\text{max}}$
and comparing with the results of a quasi-Newton method (QN).
%The errors with respect to the true state are quantified using the square of the regular Euclidean norm $\left\Vert \cdot\right\Vert _{{2}}^2$ and the maximum norm $\left\Vert \cdot\right\Vert _{\infty}$. The data is given in pairs, giving the accuracy of Newton's and the newly proposed solution method, in order. The QN method completes the task in under 120 seconds while the MIQCQP formulation is given a stop time $t_{\max} = 10 s$.
}
\label{tab_solutions_on_ieee_instances_scaling}
\begin{tabular}{|p{1.0cm}|r|r|r|r|r|r|r|r|r|r|r|r|r}
 \toprule
  & & & & \multicolumn{4}{|c|}{QN} & \multicolumn{4}{|c|}{MIQCQP}\\
 Case & buses & branches & $N$ & $cost(v_\text{est})$ & $d_2(v_\text{est})$ & $d_\infty(v_\text{est})$ & Runtime [s] & $cost(v_\text{est})$ & $d_2(v_\text{est})$ & $d_\infty(v_\text{est})$  & Runtime [s] \\
 \midrule
%  case14 & 14 & 20 & 82 & 0.00 & 0.00 & 0.05 & 7 & 0.00& 0.04 & 0.08 & $4.3\times 10^{4}$ \\
%  case39 & 39 & 46 & 209 & 234.51& 97.06 & 2.57 & 16 & 59.3 & 107.08 & 2.09 & $4.3\times 10^{4}$ \\
%  case57 & 57 & 80 & 331 & 89.32& 38.26 & 2.19 & 26 & 5.57& 63.23 & 1.56 & $4.3\times 10^{4}$ \\
% case89 & 89 & 210 & 687 & 9720 & 214.06 & 3.49 & 69 & 482 & 33.13 & 2.52 & $4.3\times 10^{4}$\\
% case118 & 118 & 186 & 726 & 1046 & 223.14 & 2.17 & 70 & 384& 230.2 & 2.06 &  $4.3\times 10^{4}$\\
% case145 & 145 & 453 & 1341 & $8.69\times 10^5$& 457.14 & 4.37 & 180 & $2.30\times 10^5$& 386 & 5.27 & $4.3\times 10^{4}$\\
% case300 & 300 & 411 & 1722 & $1.53\times 10^{5}$& 1272 & 4.41 & 258 & 458 & 954 & 2.22 & $4.3\times 10^{4}$ \\
% case2869 & 2869 & 4582 & 17771 & $3.3\times 10^{9}$& 14564 & 5.94 & $2.7\times 10^{4}$ & $2.16\times 10^{6}$ & 4375 & 4.01 & $2.52\times 10^{5}$ \\
% \midrule
case14 & 14 & 20 & 82 & 0.00 & 0.00 & 0.00 & 13.72 & 0.00 & 0.00 & 0.00 & 3124.26 \\
case30 & 30 & 41 & 172 & 0.00 & 0.00 & 0.00 & 80.16 & 15.65 & 89.05 & 2.12 & 43227.64 \\
case39 & 39 & 46 & 209 & 0.01 & 116.21 & 1.85 & 82.73 & 103.80 & 113.28 & 2.08 & 43220.69 \\
case57 & 57 & 80 & 331 & 0.00 & 0.00 & 0.00 & 552.07 & 4.43 & 30.48 & 1.28 & 43207.77 \\
case89 & 89 & 210 & 687 & 96.06 & 60.03 & 1.30 & 10301.04 & 482.09 & 33.12 & 2.52 & 43215.95 \\
case118 & 118 & 186 & 726 & 0.00 & 194.83 & 1.39 & 5508.65 & 446.59 & 260.39 & 2.19 & 43218.97 \\
case145 & 145 & 453 & 1341 & 1190.69 & 352.65 & 2.71 & 43211.64 & 7163.93 & 337.97 & 2.27 & 43210.00 \\
case300 & 300 & 411 & 1722 & 691.78 & 459.16 & 2.91 & 43210.62 & 457.94 & 954.37 & 2.22 & 43216.56 \\
%case2869 & 2869 & 4582 & 17771 & 2347308.67 & 5960.71 & 2.08 & 43422.30 & 2048806.53 & 6076.70 & 12.11 & 43214.33 \\

\bottomrule
\end{tabular} 
\end{table*}

In the following, we discuss multiple aspects of solving the state estimation problem, which lead to better accuracy and robustness of the result. 

First, the non-convex nature of the state space of the state estimation problem is highlighted as the quasi-Newton method is shown to not arrive at the global minimum for a simple example on two buses. Second, we show that the global solution can be reliably found using the formulation (\ref{eq_objective_of_poly_opt}--\ref{eq_pop_formulation_0_end}) using Gurobi. Using the same solver, we discuss scaling for this rudimentary implementation on IEEE test cases~\cite{matpower}. Lastly, we give a couple of simple examples in which the robust selection of measurements is applied.

\begin{table*}
\centering
\caption{Estimating the states of IEEE test cases with Gaussian noise
%using different stopping times $t_{\text{max}}$ 
and comparing with the results of a quasi-Newton method (QN).
%The errors with respect to the true state are quantified using the square of the regular Euclidean norm $\left\Vert \cdot\right\Vert _{{2}}^2$ and the maximum norm $\left\Vert \cdot\right\Vert _{\infty}$. The data is given in pairs, giving the accuracy of Newton's and the newly proposed solution method, in order. The QN method completes the task in under 120 seconds while the MIQCQP formulation is given a stop time $t_{\max} = 10 s$.
}
\label{tab_solutions_on_ieee_instances_scaling_with_noise}
\begin{tabular}{|p{1.0cm}|r|r|r|r|r|r|r|r|r|r|r|r|r}
%\begin{tabular}{|p{1.0cm}|p{0.5cm}|p{0.9cm}|p{1.5 cm}|p{1.3cm}|p{1.0cm}|p{0.8cm}|p{1.6cm}|p{1.3cm}|p{0.8cm}|p{1.0cm}|p{1.6cm}|p{0.8cm}|p{1.6cm}}
 \toprule
  & & & & \multicolumn{4}{|c|}{QN} & \multicolumn{4}{|c|}{MIQCQP}\\
\multicolumn{1}{|c|}{Case} & \multicolumn{1}{|c|}{buses} & \multicolumn{1}{|c|}{branches} & \multicolumn{1}{|c|}{$N$} & \multicolumn{1}{|c|}{cost$(v_\text{est})$} & \multicolumn{1}{|c|}{$d_2(v_\text{est})$} & \multicolumn{1}{|c|}{$d_\infty(v_\text{est})$} & \multicolumn{1}{|c|}{Runtime [s]} & \multicolumn{1}{|c|}{cost$(v_\text{est})$} & \multicolumn{1}{|c|}{$d_2(v_\text{est})$} & \multicolumn{1}{|c|}{$d_\infty(v_\text{est})$}  & \multicolumn{1}{|c|}{Runtime [s]} \\
 \midrule
%  case14 & 14 & 20 & 82 & 4.15 & 0.84 & 0.72 & 7 & 4.15& 0.84 & 0.72 & $4.3\times 10^{4}$ \\
%  case39 & 39 & 46 & 209 & 341.6& 94.9 & 2.49 & 16 & 162.7 & 153.87 & 2.17 & $4.3\times 10^{4}$ \\
%  case57 & 57 & 80 & 331 & 94.17& 40.41 & 2.16 & 26 & 21.75& 119.21 & 2.07 & $4.3\times 10^{4}$ \\
% case89 & 89 & 210 & 687 & 9612 & 210.66 & 3.48 & 71 & 482 & 30.67 & 2.53 & $4.3\times 10^{4}$\\
% case118 & 118 & 186 & 726 & 1244 & 205.85 & 2.18 & 70 & 335.27& 219.52 & 2.14 &  $4.3\times 10^{4}$\\
% case145 & 145 & 453 & 1341 & $1.13\times10^6 $& 469.39 & 4.37 & 175 & $1.08\times 10^5$& 486.27 & 2.74 & $4.3\times 10^{4}$\\
% case300 & 300 & 411 & 1722 & $1.46\times 10^{5}$& 1255 & 4.37 & 252 & 693.36 & 565.09 & 2.52 & $4.3\times 10^{4}$ \\
% case2869 & 2869 & 4582 & 17771 & $3.36\times 10^9$ & 14564 & 5.94 & $2.64\times 10^{4}$ & $1.54\times10^6$ & 4602 & 7.24 & $2.52\times 10^{5}$ \\
% \midrule

case14 & 14 & 20 & 82 & 4.15 & 0.84 & 0.72 & 15.03 & 4.15 & 0.84 & 0.72 & 43261.63 \\
case30 & 30 & 41 & 172 & 2.21 & 3.69 & 0.58 & 74.09 & 9.51 & 63.17 & 2.13 & 43221.81 \\
case39 & 39 & 46 & 209 & 162.72 & 136.98 & 2.18 & 228.92 & 162.71 & 153.87 & 2.17 & 43238.86 \\
case57 & 57 & 80 & 331 & 12.46 & 5.44 & 0.88 & 339.74 & 23.44 & 146.12 & 2.11 & 43222.82 \\
case89 & 89 & 210 & 687 & 336.31 & 43.29 & 1.99 & 9403.88 & 727.88 & 30.67 & 2.53 & 43213.34 \\
case118 & 118 & 186 & 726 & 46.97 & 171.22 & 1.60 & 2848.69 & 46.97 & 171.09 & 1.60 & 43215.05 \\
case145 & 145 & 453 & 1341 & 108189.87 & 375.29 & 2.93 & 43211.95 & 108821.06 & 490.83 & 2.75 & 43211.36 \\
case300 & 300 & 411 & 1722 & 1262.66 & 459.00 & 2.91 & 43214.65 & 1508.82 & 502.59 & 3.23 & 43208.93 \\
%case2869 & 2869 & 4582 & 17771 & 2410500.40 & 5960.83 & 2.08 & 50922.78 & 1630566.83 & 6103.16 & 10.00 & 43213.52 \\

\bottomrule
\end{tabular} 
\end{table*}

\subsection{Recovering the Exact Global Minimum of a Two Bus Systems}

% Mnoho lokalnich "spurious local optima" \cite{zhang2018spurious}
We present an example~\cite{zhang2018spurious} of a simple two-bus power system, in which the non-linear least squares problem has multiple local minima.
We consider a system with two buses and a single line that connects them. Line admittance is $y_{12}=y_{21}=\frac{1}{0.01+0.1\iu}$ per unit.
We set bus 1 as the reference bus and load bus 2 with power load $2+1\iu$.
The true state of the system is $v=(v_1, v_2)=[1, 0.806-0.19\iu]$.
Thus, $v$ minimizes the least squares problem with zero-error measurements and attains the optimum of $0$.

We consider the following measurements
$$h_1(v)=v^*_1v_1=\rvert v_1\lvert^2$$
$$h_2(v)=\Re[y^*_{21}(v_2-v_1)^*v_2]=p_2$$
$$h_3(v)=\Im[y^*_{21}(v_2-v_1)^*v_2]=q_2$$
$$h_4(v)=\Re[y^*_{12}(v_1-v_2)^*v_1]=p_1$$

and assume no measurement selection (all of the $\gamma$ and $\beta$ variables are set to 1) and a cost function only containing the relevant measurements with coefficients set to unity.

Using a symbolic solver, we can see the following critical points:
\begin{equation}
    \begin{pmatrix}v_1\\v_2\\\theta_2\end{pmatrix}\in\left\{\begin{pmatrix}1\\0.829\\-13.2^{\circ}\end{pmatrix},
    \begin{pmatrix}0.870\\0.345\\-35.7^{\circ}\end{pmatrix},
    \begin{pmatrix}0.846\\0.401\\-32.0^{\circ}\end{pmatrix},
    \begin{pmatrix}0\\0\\0\end{pmatrix} \right\}
\end{equation}

The above points have the following the least squares cost function values:

$$\{0, 0.11183, 0.11299, 10.297 \}.$$

The first value is the global minimum corresponding to the true state. The second point is a local minimum. The third is a saddle point. The last is a local maximum.
We can see that there is a large difference in the estimations of $v_2$ and $\theta_2$ corresponding to the two local minima.

The reason for this multiplicity of minima is nonconvexity, which has been well documented \cite{lesieutre2005convexity,lesieutre2011examining,bukhsh2013local,molzahn2016convex,molzahn2016visualizing}.
%In Figure~\ref{fig:rectangular}, we demonstrate the example of \cite{molzahn2015semidefinite}.
We refer to \cite{molzahn2019survey} for an exhaustive survey.

This multiplicity leads to any local method being dependent on the initialization, which is highly undesirable. However, using a global solution method based on (\ref{eq_objective_of_poly_opt}--\ref{eq_pop_formulation_0_end}) allows us to recover the global minimum of the problem and recover the state of the system up to numerical precision.

%\begin{figure}
%    \centering
%    \includegraphics[scale=0.4]{img/contourf_plot.png}
%    \caption{Contour plot of the least square function}
%    \label{fig:enter-label}
%\end{figure}

%\begin{figure}
%    \centering
%    \includegraphics[scale=0.4]{img/spurios_plot.png}
%    \caption{3D plot of the least square function}
%    \label{fig:enter-label}
%\end{figure}

%\paragraph{Example of a simple power system with a non-convex state estimation least square problem}

\subsection{Estimating the Steady State on IEEE Test Cases}
In the following, we present results of our novel mixed-integer quadratically constrained quadratic program (MIQCQP) formulation on variants of commonly used IEEE test cases~\cite{matpower} and compare the results with those of a quasi-Newton method (L-BFGS-B, \cite{zhu1997algorithm}) implemented in SciPy \cite{2020SciPy-NMeth}\footnote{https://docs.scipy.org/doc/scipy/reference/optimize.minimize-lbfgsb.html} with default setting. %to get the first returns on accuracy and running times, which would be applicable in close to real world scenarios. 
% https://pandapower.readthedocs.io/en/v2.2.1/estimation.html
In particular, we present results for two variants of each instance, differing in the amounts of noise applied to the measurements.
Each of the experiments was run on a computing cluster with AMD EPYC 7543 cpus. We set a limit of 64~GB RAM and four cores per task.

We note that no warm-start or fine-tuning of the constants in the formulation or the solver to facilitate the speedup for the specific instances has been performed for the MIQCQP formulation, i.e., only plain-vanilla Gurobi with quadratic programming (QP) relaxations has been used. This gives hope of dramatic speedups by fine-tuning the solver, utilizing SDP relaxations, and applying mixed-integer decompositions. Some of these options are discussed in Sections \ref{sec_alternative_algs} and \ref{sec_time_var_SDP}.
%As a gradient descent method (GD) we used an algorithm L-BFGS~\cite{Liu1989OnTL} from the family of quasi-Newton methods. 
Within the quasi-Newton method, we start with an initial guess equal to $1+\iu0$ for every bus. The algorithm stops when either i) $\left|f^{k}-f^{k+1}\right|/\max\{|f^k|,|f^{k+1}|,1\}\leq 10^{-6}$ or ii) $\textrm{max}\{|g_i|;i = 1, \dots, 2n-1\} \leq 10^{-6}$ where $f_k$ is the solution in $k$-th iteration and $g_i$ is $i$-th coordinate of the gradient.

Tables \ref{tab_solutions_on_ieee_instances_scaling} and \ref{tab_solutions_on_ieee_instances_scaling_with_noise} present the results of the scaling experiments for IEEE test cases with and without noise, respectively.
The noise was applied in the following way. For each measurement, we computed the true measurement value $m$ based on the true state and add a noise sampled from Gaussian distribution with mean $0$ and variance equal to $0.1m$.
The quality of the results is evaluated based on the comparison with the true state in terms of voltages (which is known for the test cases) using the following two distances:
\begin{equation}
    d_2(v_\textrm{est})=\Vert v - v_\textrm{est} \Vert_2^2,\quad d_\infty(v_\textrm{est})=\Vert v - v_\textrm{est} \Vert_\infty,
\end{equation}
where $v$ is the true state.
The above norms are defined as follows
\begin{equation}
\left\Vert v\right\Vert _{2}^{2}=\sum_{i=1}^{n}\left|v_{i}\right|^{2},\quad\left\Vert v\right\Vert _{\infty}=\underset{i\in\left\{ 1,2,\ldots,n\right\} }{\max}\left|v_{i}\right|.
\end{equation}
The $cost(v_\text{est})$ equals to the expression~\eqref{eq_objective_of_poly_opt} evaluated at the estimated state $v_\text{est}$. Note that, in the case if all coefficients $C^*_*$, $\beta^*_*$, $\gamma_*$ equal 1, then $cost(v_\text{est})$ is equal to root mean square error (RMSE) squared with respect to the estimated values of measurements, which are derived from the estimated state $v_\text{est}$.

The experiments demonstrate that the quality of solutions obtained using the quasi-Newton method becomes progressively worse as the size of the grid increases. 
%However, the solutions obtained using the newly proposed formulation maintain quality when scaling even when the solution time is limited to 10 seconds. 
Note that, if the computation time was not limited, the MIQCQP solver would continue and eventually find the true global minimum. 
%This is, however, not desirable and one aims to compare the solution qualities for computation times in the order of seconds, since the end result is usually only useful for a limited time. 

We can see that MIQCQP method performs better, in terms of $d_2(v_\text{est})$, on the problems with a larger number of buses. MIQCQP performs always better in terms of $cost(v_\text{est})$. On some problems substantially. %Note that minimizing $cost(v_\text{est})$ is a necessary condition in order to minimize $d_2(v_\text{est})$. 
A drawback of MIQCQP method is a longer running time. Which we believe can be decreased substantially by fine tuning MIQCQP method. 

%Since the voltage of the true state is expressed in the per unit system, the newly proposed method admits a worst-case error around $1$-$2$ times the state size, while quasi-Newton methods deliver a maximal error that is more than $4$ times the state size for larger problems. Furthermore, cumulative errors (expressed as $d_2(v_\text{est})$) are also higher for quasi-Newton method, reaching up to 2x and 3x in the case of 118 and 145 bus systems compared to the newly proposed method. 

Since our method guarantees global convergence, these figures can be further improved given additional time, tweaking the formulation, improving the cutting plane algorithm or improving hardware, while in the case of quasi-Newton method, no improvement can be made by committing further resources, such as improved hardware or higher time limit.

\subsection{Robust Statistics}\label{sec_robust_stat}
In this section, we demonstrate the utility of robust statistics for the elimination of outlier data. The data is generated subject to two types of  Gaussian noise. First, we choose fault probability $p_f$ to be either $0.1$ or $0.01$. Then we uniformly with probability $p_f$ select faulty measurements. Let $m$ be the measurement value without noise. Then we add Gaussian noise with mean $0$ and variance either $0.1m$ or $100m$ depending if the measurement was selected as non-faulty or faulty, respectively.

The number of required measurement ($d$ in Equation~\ref{eq_robust_start}) is chosen either $0.9$ or $0.99$ times the total number of measurements, if $p_f$ equals $0.1$ or $0.01$, respectively.

Under these circumstances, we compare the results of our method with the quasi-Newton method of SciPy and the non-robust variant of our formulation.

The results are documented in Table \ref{tab_robust_results}. We can see that the robust formulation performs substantially better, in terms of $d_2(v_\text{est})$, than the quasi-Newton method or the non-robust variant.

\begin{table*}
\centering
\caption{A demonstration of the effects of the use of robust statistics. The mean and standard deviation (std) over 10 runs,
where noise is specified by the probability of applying a grave error, whose amplitude is 1000x greater then the usual noise.
} \label{tab_robust_results}
\begin{tabular}{|p{0.8cm}|r|r|r|r|r|r|r|r|r|r|}
 \toprule
  & & \multicolumn{3}{|c|}{QN} & \multicolumn{3}{|c|}{Standard MIQCQP}& \multicolumn{3}{|c|}{Robust MIQCQP}
  \\
  
 Case & Fault & \multicolumn{2}{|c|}{$d_2(v_\text{est})$} & Runtime &\multicolumn{2}{|c|}{$d_2(v_\text{est})$} & Runtime &\multicolumn{2}{|c|}{$d_2(v_\text{est})$}& Runtime
 \\
&  probability & mean & std & mean [s] & mean & std & mean [s] & mean &  std  & mean [s]
 \\
 \midrule
 %case14 & 0.01 & 138 & 20930 & 119 & 25022 & 37 & 1368  \\
 %case14& 0.05 & 368 & 168838 & 360 & 163293 & 47 & 3882  \\

case14 & 0.01 & 227.78 & 445.85 & 7.79 & 131.99 & 175.00 & 43248.27 & 1.97 & 1.52 & 43235.31 \\
case30 & 0.01 & 318.75 & 511.54 & 45.78 & 276.13 & 487.05 & 43228.04 & 12.07 & 9.08 & 43207.63 \\
case39 & 0.01 & 2396.88 & 4532.82 & 80.26 & 1821.78 & 3947.85 & 43229.35 & 111.34 & 46.48 & 43220.79 \\
case57 & 0.01 & 491.24 & 551.61 & 232.71 & 539.55 & 534.68 & 43226.60 & 82.02 & 22.52 & 43207.50 \\

% case14 & 0.1 & 566.06 & 484.38 & 4.56 & 567.66 & 469.51 & 43233.32 & 47.80 & 63.70 & 43204.91 \\
% case30 & 0.1 & 1089.49 & 955.80 & 13.30 & 1106.35 & 892.73 & 43227.25 & 51.59 & 10.58 & 43202.21 \\
% case39 & 0.1 & 5840.19 & 4890.41 & 17.52 & 6205.94 & 5055.74 & 43237.14 & 1454.20 & 4277.30 & 43202.68 \\
% case57 & 0.1 & 1604.18 & 648.48 & 28.90 & 1801.30 & 666.48 & 43220.04 & 106.74 & 61.54 & 43203.58 \\
%case89 & 0.1 & 1857.36 & 990.58 & 73.42 & 4148.94 & 1771.89 & 34620.28 & 144.77 & 42.09 & 43205.37 \\
%case118 & 0.1 & 3390.75 & 1053.28 & 69.59 & 4126.06 & 1923.75 & 38898.95 & 200.57 & 94.45 & 43204.27 \\
%case145 & 0.1 & 875.76 & 101.50 & 172.52 & 3136.71 & 2129.89 & 43206.99 & 404.24 & 168.88 & 43204.66 \\
%case300 & 0.1 & 2120.30 & 717.29 & 257.62 & 13485.65 & 4160.64 & 43210.81 & 706.16 & 209.66 & 43204.89 \\

% case14& 0.1 & 566.06 & 234.62 & 6 & 567.66 & 232.69 &7200 & 55.64 & 6.60 &7200 \\
% case30& 0.1 & 1089.49 & 913.56 & 14  & 1106.34 & 841.24 & 7200 & 99.85 & 8.44 & 7200  \\
% case39& 0.1 & 5840.19 & $2.39\times10^7$ & 18 & 6101.21 & $2.75\times10^7$ & 7200 & 199.86 & 27372 & 7200  \\
%case57 & 0.1 & 1604.18 & fill & 29 & 1798.59 & fill & 7200 & 195.21 & fill & 7200 \\
%case89 & 0.1 & 1857.36 & fill & 80  & 4144.3 & fill & 7200 & 236.72 & fill & 7200  \\
%case118 & 0.1 & 3390.75 & fill & 85  & 4115.77 & fill & 7200 & 714.56 & fill & 7200  \\

case14 & 0.10 & 566.67 & 483.49 & 29.49 & 567.66 & 465.45 & 43232.17 & 43.99 & 54.78 & 43205.90 \\
case30 & 0.10 & 1109.04 & 952.78 & 46.40 & 1106.35 & 885.00 & 43227.66 & 67.37 & 19.03 & 43203.37 \\
case39 & 0.10 & 6473.25 & 5301.31 & 86.62 & 6199.83 & 5012.91 & 43236.14 & 777.72 & 3024.57 & 43203.62 \\
case57 & 0.10 & 1681.75 & 649.46 & 179.01 & 1788.87 & 650.08 & 43218.39 & 89.08 & 46.54 & 43203.24 \\

% case14 & 0.01 & 227.64 & 445.86 & 17.11 & 131.99 & 175.00 & 43248.27 & 1.97 & 1.52 & 43235.31 \\
% case30 & 0.01 & 304.46 & 477.22 & 24.82 & 276.13 & 487.05 & 43228.04 & 12.07 & 9.08 & 43207.63 \\
% case39 & 0.01 & 1907.98 & 3372.90 & 27.23 & 1821.78 & 3947.85 & 43229.35 & 111.34 & 46.48 & 43220.79 \\
% case57 & 0.01 & 479.80 & 538.19 & 37.44 & 539.55 & 534.68 & 43226.60 & 82.02 & 22.52 & 43207.50 \\
% case89 & 0.01 & 873.69 & 1172.58 & 81.05 & 1214.40 & 1490.26 & 38919.34 & 118.47 & 6.92 & 43211.38 \\
% case118 & 0.01 & 933.50 & 646.31 & 80.75 & 1109.18 & 757.74 & 43217.73 & 156.20 & 23.08 & 43204.46 \\
% case145 & 0.01 & 504.75 & 55.82 & 180.78 & 608.28 & 280.09 & 43218.33 & 216.25 & 34.67 & 43219.11 \\
% case300 & 0.01 & 1324.37 & 107.84 & 263.41 & 2137.32 & 746.86 & 43226.91 & 591.24 & 581.00 & 43220.35 \\

\bottomrule
\end{tabular} 
\end{table*}

%\begin{figure}
%    \centering    \includegraphics[width=0.4\linewidth]{img/case14var10.png}
%\includegraphics[width=0.4\linewidth]{img/case14var100.png}
%    \caption{Case14}
%    \label{fig:enter-label}
%\end{figure}

\section{Tracking Trajectories of Time-Varying Problems}\label{sec_alternative_algs}

The MIQCQP formulation of Section \ref{sec_num_study} can be passed to off-the-shelf solvers (such as Gurobi of the previous section) to solve the problem globally with guarantees. Although this is a very robust method, other methods may be more desirable in real-time optimization practice.
In the following, one of these methods is described. 

Polynomial optimization problems (POP) have been studied for multiple decades \cite{anjos2011handbook}, and multiple solution methods exist, often involving semidefinite programming (SDP) relaxations. 
Recently, solution tracking in SDP relaxations has been developed \cite{liu2017,Bellon_2024,bellon2024time} as a method to provide near real-time optimization with guarantees on the distance from the true trajectory of optima by leveraging knowledge about previous solutions.
 
\subsection{The Derivation of SDP Relaxation of the State Estimation in Polynomial Formulation}
SDP relaxations of (\ref{eq_objective_of_poly_opt}-\ref{eq_pop_formulation_0_end}) based on the method of moments \cite{Lessere04} will be briefly sketched. To write out the general form of the relaxation, a couple of definitions need to be stated.
\begin{definition}[Moment matrix of order $d$]\label{eq_moment_matrix}
Let $d,n \in \mathbb{N}$ and consider a finite sequence $y=\left(y_\alpha\vert \alpha\in\mathbb{N}^n, \left\vert\alpha\right\vert\leq2d\right)$, where $\alpha=(\alpha_1,\dots,\alpha_n)$ is a multi-index and $\left\vert\alpha\right\vert=\alpha_1+\dots+\alpha_n$, which defines the following linear functional on the space of polynomials up to degree $2d$:
\begin{equation}
L_{y}\left(p\right)=\sum_{\lvert\alpha\rvert\leq 2d}p_\alpha y_\alpha 
\end{equation}
where $p=\sum_{\lvert\alpha\rvert\leq 2d}p_\alpha x^\alpha$, where $x^\alpha=x_1^{\alpha_1}x_2^{\alpha_2}\cdots x_n^{\alpha_n}$. Then the moment matrix $M_{d}\left(y\right)$ associated with $y$ is defined as:
\begin{equation}
M_{d}\left(y\right)_{\alpha\beta}=L_{y}\left(x^\alpha x^\beta\right)=y_{\alpha+\beta},
\end{equation}
where $\alpha,\beta\in\mathbb{N}^n$, $\lvert\alpha\rvert\leq d$, $\lvert\beta\rvert\leq d$.
\end{definition}
It is clear to see that the moment matrix is always symmetric.
\begin{definition}[Localizing matrix of order $d$ with respect to polynomial $a(x)$]
Suppose $r\in\mathbb{N}_0$, $d, n\in \mathbb{N}$ are such that $r < d$, $a(x)=\sum_{\lvert\alpha\rvert\leq r}a_{\alpha}x^{\alpha}$ and let $y$ be as in Definition~\ref{eq_moment_matrix}.

\begin{align}
\begin{split}
M_{d}\left(a(x),y\right)_{\alpha\beta}=L_{y}\left(x^\alpha x^\beta a(x)\right)=\sum_{\vert\gamma\rvert\leq 2d - \lvert\alpha\rvert - \lvert\beta\rvert}a_{\gamma}y_{\gamma+\alpha+\beta},
\end{split}
\end{align}
where $\alpha,\beta, \gamma \in\mathbb{N}^n$, $\lvert\alpha\rvert\leq d$, $\lvert\beta\rvert\leq d$.
\end{definition}
Using the aforementioned concepts and indexing all the polynomial inequalities in (\ref{eq_objective_of_poly_opt}-\ref{eq_pop_formulation_0_end}) using the index set $I$ resulting in $g_{i} \geq 0$, one can write the moment SDP relaxation as
\begin{align}
    \begin{split}\label{eq_moment_problem}
\rho_{d}^{\text{mom}} &=\underset{y\in\mathbb{R}^{s\left(2d\right)}}{\inf}\sum_{\vert\alpha\rvert\leq 2d}f_{\alpha}y_{\alpha}, \\
\text{s.t. } &y_{0\cdots0}=1,M_{d}\left(y\right)\succeq0, \\
& M_{d-d_{i}}\left(g_{i}, y\right)\succeq0,\quad i\in I,
    \end{split}
\end{align}
where $y=\left(y_\alpha\vert \alpha\in\mathbb{N}^n, \left\vert\alpha\right\vert\leq2d\right)$, $d_{i}=\left\lceil\frac{1}{2} \deg\left[g_{i}\right]\right\rceil  $, $s=s\left(d\right)$ is the number of monomials up to degree $d$ and $f_{\alpha}$ are the coefficients of the polynomial cost function $f$. Note that the constraints in \eqref{eq_moment_problem} are a rewrite of the original constraints and the cost function arises by reformulating the original cost function in terms of moments of measures.

%To gain more insight on the theoretical properties of the problem in terms of the dynamic SDP it is purposeful to resolve some of the constraints in a more granular way. 
%In particular, additional slack can be shown in the equalities defining $p_k, q_k, p_{f,l}, p_{t,l}, q_{f,l}, q_{t,l}$.

\subsection{The Derivation of the SDP Relaxation for Relevant Forms of Constraints}\label{app_derivation_of_SDP}
Suppose that $n \in \mathbb{N}$, $v \in \mathbb{C}^{n}$, using this setting we resolve the right-hand side of $p_{k}=\Tr\left(Y_{k,p}vv^{*}\right)$. Unravelling all the relevant definitions yields
\begin{align}
\begin{split}
\Tr\left(\frac{1}{2}(Y^{*}E_{k}vv^{*}+E_{k}Yvv^{*})\right)= \\
= \frac{1}{2}(\Tr(Y^{*}E_{k}vv^{*})+\Tr(E_{k}Yvv^{*}))= \\
=\ 
\frac{1}{2}\left[\sum^n_{i=1}\overline{Y_{ki}}\overline{v_{i}}v_k + \sum^n_{i=1}Y_{ki}v_{i}\overline{v_k}  \right]=\\
=\frac{1}{2}(s_k+\overline{s_k})=p_k
\end{split}    
\end{align}

What follows are some computations that lead to the explicit form of the SDP relaxation. To resolve many of the equalities, we need the following, which follows by direct computation.
\begin{align}
    \begin{split}
        \left(a_{1}-\iu b_{1}\right)\left(a_{2}-\iu b_{2}\right)\left(a_{3}+\iu b_{3}\right) \\
        = a_{1}a_{2}a_{3}-b_{1}b_{2}a_{3}+a_{1}b_{2}b_{3}+b_{1}a_{2}b_{3} \\
        +\iu\left(a_{1}a_{2}b_{3}-b_{1}b_{2}b_{3}-a_{1}b_{2}a_{3}-b_{1}a_{2}a_{3}\right) \label{eq_simple_complex_formula}
    \end{split}
\end{align}

Due to the symmetry of the constraints, we provide a detailed derivation only considering the potential and active power variables. Define the moment vector as
\begin{align}
\begin{split}
v_{d}\left(v,p\right)=\left(1,v_{1},v_{2},\ldots,v_{2n},p_{1},\ldots,p_{n}\right.,\\
v_{1}^{2},v_{1}v_{2},\ldots,v_{1}v_{2n},v_{1}p_{1},\ldots,v_{1}p_{n},v_{2}^{2}, \\
v_{2}v_{3},\ldots,v_{2}p_{n},v_{3}^{2},\ldots,p_{n}^{2},\ldots,\left.p_{n}^{d}\right).
\end{split}
\end{align}

Setting $d=2$ (note that $d = 1$ leads to the localization matrix being trivial), set sequence $y_d(v,p)$ satisfying $$y_{(i,j,k)}=v_iv_jp_k,$$ where $(i,j,k)$ is a multi-index that has ones on the $i$-th, $j$-th and $k$-th positions and zeros everywhere else. Define
\begin{equation}
g_{p,k}(v, p_k)=\Tr\left(Y_{k,p}vv^{*}\right)-p_{k}\text{ for each }k\in I. \label{eq_constraints_for_pop}
\end{equation}
and
$d_{p,k} = \left\lceil\frac{1}{2} \deg\left[g_{p,k}\right]\right\rceil  = 1$ for every $k$. Before resolving $M_{d-d_{p,k}}\left(g_{p,k},y\right)=M_{1}\left(g_{p,k},y\right)$ compute
\begin{align}
\begin{split}
p_{k}= & \sum_{i=1}^{n}\left[Y_{ki}^{\Re}v_{i}^{\Re}v_{k}^{\Re}-Y_{ki}^{\Im}v_{i}^{\Im}v_{k}^{\Re}\right. \\
& \left.+Y_{ki}^{\Re}v_{i}^{\Im}v_{k}^{\Im}+Y_{ki}^{\Im}v_{i}^{\Re}v_{k}^{\Im}\right]
\end{split},
\end{align}
where \eqref{eq_simple_complex_formula} was used. The moment matrix then equals 
\begin{align}
\begin{split}
M_{1}\left(g_{p,k},y_d(v,p)\right)_{\alpha\beta}=\\
=\sum_{i=1}^{n}\left[Y_{ki}^{\Re}y_{\left(i,k\right)+\alpha+\beta}-Y_{ki}^{\Im}y_{\left(i+n,k\right)+\alpha+\beta} +\right. \\ 
\left.+\,Y_{ki}^{\Re}y_{\left(i+n,k+n\right)+\alpha+\beta}+Y_{ki}^{\Im}y_{\left(i,k+n\right)+\alpha+\beta}\right] - \\
-\,y_{\left(2n+k\right)+\alpha+\beta},\label{eq_definition_of_moment_order_1}
\end{split}    
\end{align}

% \begin{align}
% \begin{split}
% M_{1}\left(g_{p_{f,k}},y\right)\left(l,r\right)=\sum_{i=1}^{n}\left[Y_{ki}^{\Re}m_{\left(i,k\right)+l+r}-Y_{ki}^{\Im}m_{\left(i+n,k\right)+l+r}\right. \\ 
% +Y_{ki}^{\Re}m_{\left(i+n,k+n\right)+l+r}+Y_{ki}^{\Im}m_{\left(i,k+n\right)+l+r} \\
% \left.+m_{\left(2n+k\right)+l+r}\right],\label{eq_definition_of_moment_order_1}
% \end{split}    
% \end{align}

where $\left(i,k\right)$ is used to denote a multi-index that has ones on the $i$-th and $k$-th positions and zeros everywhere else. What follows is a consequence of imposing the equality $g_k = 0$. 

\begin{lemma}\label{lem_imposing_both_ineqs}
Let $M_{1}\left(g_{p,k},y\right)$ be defined by \eqref{eq_definition_of_moment_order_1} and suppose that 
\begin{equation}
M_{1}\left(g_{p,k},y\right)\succeq0\text{ and }M_{1}\left(g_{p,k},y\right)\preceq0,\label{eq_both_conditions}
\end{equation}
which corresponds to imposing $g_k = 0$ in the original polynomial optimization problem. Then the matrix $M_{1}\left(g_{p,k},y\right)$ has only zero eigenvalues and thus 
\begin{equation}
M_{1}\left(g_{p,k},y\right)_{\alpha,\beta} = 0\label{eq_moment_is_zero}
\end{equation}
for each $k \in \left\{ 1,2,\ldots,n\right\} $ and all multi-indices $\alpha$ and $\beta$.
\end{lemma}
\begin{proof}
Interpreting \eqref{eq_both_conditions} in terms of eigenvalues implies that $M_{1}\left(g_{p,k},y\right)$ must have only zero eigenvalues. The only symmetric matrix that has all zero eigenvalues is the zero matrix leading to \eqref{eq_moment_is_zero}.
\end{proof}
An immediate consequence of the preceding lemma is the following claim.
\begin{lemma}\label{lem_slater_fail}
If $M_{1}\left(g_{p,k},y\right) = 0$ and the admittance matrix is non-zero, the interior of the set of feasible solutions is the empty set i.e. Slater's condition does not hold for any feasible set containing constraints implying this condition.
\end{lemma}
\begin{proof}
Since \eqref{eq_moment_is_zero}  is equivalent to imposing at least one non-trivial element-wise equality constraint on the space of positive symmetric moment matrices (due to the admittance matrix not being zero), the co-dimension of the feasible set is at least $1$. 
\end{proof}
Using the preceding Lemma, it is clear to see that a reformulation of the original polynomial optimization problem is for the tracking of the SDP trajectories to be meaningful. Recalling the constraints \eqref{eq_constraints_for_pop} define a new set of constraints as
\begin{equation}
g_{k}\leq\delta,\quad g_{k}\geq\delta,    
\end{equation}
where $\delta > 0$ is an arbitrary slack variable that prevents the pathological situation described in Lemmas \ref{lem_imposing_both_ineqs} and \ref{lem_slater_fail}.
%RTE is a major player in the energy transition and pursues the mission of guaranteeing, 24 hours a day, a safe, economical and environmentally friendly power supply. Open Source Software is playing a growing role in the success of these missions: it accelerates development roadmaps, reduces costs through shared efforts, increases code quality and stimulates innovation. Peter presented how open RAO, one of the major components of the LF Energy Foundation PowSyBl project, optimizes the operations of the European power network, while promoting strong interoperability among transmission systems operators and their regional coordination centres. He also presented how PowSyBl, as a collaborative open source technology, contributes to better alignment and transparency in operational processes. Jean-Baptiste and Geoffroy presented the variety of applications that can be built leveraging the PowSyBl technology, from state-of-the-art industrial-grade software for power grid operations to research and prototyping-oriented developments in Python.

%(see Appendix \ref{app_derivation_of_SDP}) is necessary in order to have a point on the interior of the feasible set (Slater's condition). This becomes important when tracking the dynamical SDP (see Section \ref{sec_time_var_SDP}).

\subsection{Positive Results for a Related Problem}

The related problem of alternating current optimal power flow (ACOPF) can be cast as a polynomial optimization and has been studied extensively \cite{liu2023optimal, Lavaei12, best19}. 

It is known that under mild assumptions there exists a $k \in \mathbb{N}$ such that the $k$-th SDP relaxation gives the exact solutuion of the problem. Furthermore, it has been shown for a particular form of the ACOPF problem that the first level of the SDP relaxation hierarchy ($k = 1$) provides the exact solution \cite{Lavaei12}.

Recent research \cite{liu2017,Bellon_2024,bellon2024time}  studied the dynamics of time-dependent semi-definite programs and time-dependent polynomial optimization problems. These techniques have been applied to the related problem of optimal power flow with considerable success \cite{liu2023optimal}. This motivates us to apply these tracking techniques to state estimation problems.

\subsection{Positive Results for the State Estimation}\label{sec_time_var_SDP}

 We first present a brief overview of the theoretical results, which are key for understanding the dynamics of time-varying SDPs, then the relevance of this theory to the specific problem of state estimation is discussed.

\subsection{An Overview}\label{sec_trajectory_tracking_SDP_genral}
In this section $m, n \in \mathbb{N}$ are arbitrary instead of having the meaning of branches and nodes. Recently, the trajectories of time-varying semi-definite programs (SDPs) have been considered \cite{Bellon_2024}. More specifically, one considers the time-varying problem:
\begin{align}
\begin{split}
\underset{X\in\mathbb{S}^{n}}{\min}\left\langle C\left(t\right),X\right\rangle \\
\text{s.t. }\mathcal{A}\left(t\right)\left[X\right]=b\left(t\right) \\
X\succeq0,\label{eq_time_dep_sdp_first}
\end{split}
\end{align}
where $\mathbb{S}^n$ denotes the set of symmetric $n$ by $n$ matrices, $b\left(t\right) \in \mathbb{R}^m$ for each $t$, $\left\langle \cdot,\cdot\right\rangle $ is an inner product on the space of matrices, $\mathcal{A}\left(t\right)\left[X\right]=\left(\left\langle A_{1}\left(t\right),X\right\rangle ,\left\langle A_{2}\left(t\right),X\right\rangle ,\ldots,\left\langle A_{m}\left(t\right),X\right\rangle \right)$ and $C\left(t\right)\in\mathbb{S}^{n}$  for each $t$ is given. The realization of \eqref{eq_time_dep_sdp_first} for the state estimation problem would be the encoding the objective function in the scalar product term $\left\langle C\left(t\right),X\right\rangle$ and the measurement variables along with Kirchhoff's and Ohm's laws being represented by $\mathcal{A}$ and $b$. The corresponding trajectories of feasible solutions, optima, and optimizers are:
\begin{align}
\begin{split}
\mathcal{P}\left(t\right)&=\left\{ X\in\mathbb{S}^{n}:\mathcal{A}\left(t\right)\left[X\right]=b\left(t\right),X\succeq0\right\} , \\
p^{*}\left(t\right)&=\underset{X\in\mathbb{S}^{n}}{\min}\left\{ \left\langle C\left(t\right),X\right\rangle :X\in\mathcal{P}\left(t\right)\right\},  \\
\mathcal{P}^{*}\left(t\right)&=\left\{ X\in\mathcal{P}\left(t\right):\left\langle C\left(t\right),X\right\rangle =p^{*}\left(t\right)\right\}, 
\end{split}    
\end{align}
respectively.
We say that the time-varying problem  \eqref{eq_time_dep_sdp_first} satisfies the basic regularity conditions if
\begin{itemize}
    \item the interior of the feasible set is non-empty i.e. if there exists a strictly feasible point (Slater's condition),
    \item if the data vectors and matrices dependent on time $C\left(t\right),\mathcal{A}\left(t\right)$ and $b\left(t\right)$ are continuous with respect to any given matrix or vector norm on the codomain (continuity is independent of norm choice),
    \item the constraints satisfy the so called linear independence constraint qualification (LICQ), which is satisfied if the matrices $A_{i}$ (forming the operator $\mathcal{A}$) are linearly independent, which makes $\mathcal{A}$ surjective.
\end{itemize}

A complete classification of trajectory points of a time-varying SDP  has been developed \cite{Bellon_2024}.
The Appendix \ref{sec_trajectory_tracking_SDP_genral},
%\footnote{See arxiv ...}, 
contains the detailed classification. Under the basic regularity conditions, the following pathological behaviors do not occur \cite{Bellon_2024}.
\begin{itemize}
    \item Any neighborhood of a point may contain many points at which multiple solutions to the problem exist (non-isolated isolated multi values).
    \item A trajectory may split into multiple trajectories at a point while preserving continuity (continuous bifurcation).
    \item Non-continuous split into multiple trajectories (irregular accumulation point).
\end{itemize}
Each of the pathological behaviors listed above would introduce major issues to any tracking algorithm that uses preceding solutions to compute the current solution within the trajectory. On the other hand, the absence of all of these behaviors (which is guaranteed by the fulfilment of the regularity conditions) allows for efficient tracking algorithms \cite{bellon2024time}. 

\subsection{Application to the Steady State Estimation Problem}

Assuming that the system evolves and is tracked at discrete times $T=\left\{ t_{0},t_{1},\ldots,t_{\text{final}}\right\} $ one can view the state estimation problem (\ref{eq_objective_of_poly_opt}-\ref{eq_pop_formulation_0_end}) as a polynomial optimization problem with a time parameter, which can be written in general as
\begin{equation}\label{eq_time_var_poly}
\underset{x\in\mathbb{R}^{N}}{\min}f^{j}\left(x\right)\text{ s.t. }g_{i}^{j}\left(x\right)\geq0,\forall i\in I,    
\end{equation}
where $j$ represents the time level index, where $f^j$ and $g_i^j$ are polynomails in $x$ for each $i$ and $j$. As was mentioned before, it is possible to find the exact solution of \eqref{eq_time_var_poly} by solving an SDP relaxation. Instead of solving each of these instances separately, one may study the trajectories of the solutions and develop tracking methods \cite{liu2023optimal,bellon2024time}.

To apply the general-purpose tracking method \cite{bellon2024time}, one needs to ensure the basic regularity condition. Slater's condition is satisfied, as suggested in Section \ref{app_derivation_of_SDP}. (This is independent of generic methods of obtaining Slater's point \cite{kungurtsev2020twostep} based on facial reduction. These may need to be applied also in the case that the matrices $A_i$ forming the operator $\mathcal{A}$ are linearly dependent \cite{kungurtsev2020twostep}.) %ensured if the operation of the network excludes discontinuous actions based on decisions taken when operational quantities exceed deadband limits. In the case of the state estimation problem, the data is can be considered 
%can be seen as a polynomial function of time, discounting any discontinuous behavior. %\textcolor{red}{Could Jakub add a comment on non-singular time? I really don't know how to characterize it in a compact way and I don't know it we can say that it holds}. 
The continuity of the data is impossible to guarantee in neither normal operations (e.g., deadband reached) nor outage scenarios (e.g., a short circuit). The tracking of the trajectory of solutions has to utilize change-point detection methods \cite{aminikhanghahi2017survey} whenever there are discontinuities.

A wide variety of tracking methods have been developed over the years. The most straightforward method would be to warm-start repeated solution of the problem for each time instant \cite{Gondzio1998WarmSO}. % 2c05efe88aef485f9afa0a8cb6d4947b
Such tracking methods typically use the interior-point method and a small number of iterations and result in solutions close to machine precision, but the savings of computation time obtained by warm-starting are modest. 

Using a path-following predictor-corrector method is another way to make use of the time-varying dynamics \cite{bellon2024time}. These methods focus on solving a convex optimization problem first, which gives a predictor of the evolution. The SDP instance at current time is used to correct the predictor. 
Another path-following method \cite{liu2023optimal}, that can be utilized even for non-convex problems under certain assumptions, is one that relies on the Polyak-Lojasiewicz (PL) inequalities \cite{Polyak63}. These inequalities give a growth condition on the gradient around local minima, which results in certain guarantees of global optimality holding when using gradient or Newton based methods. Note that there is no relationship between functions convex and functions that fulfill PL \cite{Bolte20}, although there is some indication \cite{best19} that problems with non-convex transmission constraints satisfy the PL inequality. 

\section{Conclusion}
%\textcolor{blue}{TODOS:
%\begin{itemize}
%    \item agument the conclusion,
%    \item add a small writeup about the possible implementation
%\end{itemize}
We have revisited the problem of power system steady-state estimation, considering the need for robust statistics, the non-convex transmission constraints, and the need to track trajectories of optimal solutions as closely as possible. 
Starting with a commonly utilized 
structure of the measurement chain (involving SCADA and PMU measurements), we derive appropriate noise models and propose a robust estimator in the so-called Huber model of robust statistics. We show that off-the-shelf branch-and-cut solvers can handle both the robust aspects and the rectangular power-voltage formulation of the transmission constraints in a single mixed-integer quadratically constrained quadratic program (MIQCQP).  

Our experiments on standard IEEE test cases demonstrate the benefits of an estimate in the Huber model with full-fledged non-convex transmission constraints. Specifically, cumulative errors expressed in the $d_2(v_\text{est})$ distance from the true state can be reduced by an order of magnitude on average over the tested problems, compared to the use of a standard quasi-Newton method on networks with many faulty measurements.

%The novel approach is capable of reducing RMSE of the state estimation around 40x compared to the classical quasi-Newton method as demonstrated for power system consisting of 2869 buses.

We believe that this perspective may spur further research in state estimation, whose reliability is crucial for the operations management in transmission systems. 

%\textcolor{blue}{GLOBAL TO DOS:
%\begin{itemize}
%    \item tweak the intro and abstract,
%    \item add the numerical experiments: robust related, scaling related
%    \item make sure that the section on the robustness matches the experiments in the end,
%    \item add CVs
%    \item add a conclusion!
%\end{itemize}}

%\textcolor{blue}{CVs go here}
\bibliographystyle{ieeetr}
\bibliography{ref}

\begin{thebibliography}{10}

\bibitem{prostejovsky2019future}
A.~M. Prostejovsky, C.~Brosinsky, K.~Heussen, D.~Westermann, J.~Kreusel, and
  M.~Marinelli, ``The future role of human operators in highly automated
  electric power systems,'' {\em Electric Power Systems Research}, vol.~175,
  p.~105883, 2019.

\bibitem{schweppe1970powerA}
F.~C. Schweppe and J.~Wildes, ``Power system static-state estimation, part i:
  Exact model,'' {\em IEEE Transactions on Power Apparatus and systems}, no.~1,
  pp.~120--125, 1970.

\bibitem{schweppe1970powerB}
F.~C. Schweppe and D.~B. Rom, ``Power system static-state estimation, part i:
  Approximate model,'' {\em IEEE Transactions on Power Apparatus and Systems},
  no.~1, pp.~125--130, 1970.

\bibitem{schweppe1970powerC}
F.~C. Schweppe, ``Power system static-state estimation, part iii:
  Implementation,'' {\em IEEE Transactions on Power Apparatus and systems},
  no.~1, pp.~130--135, 1970.

\bibitem{monticelli2000electric}
A.~Monticelli, ``Electric power system state estimation,'' {\em Proceedings of
  the IEEE}, vol.~88, no.~2, pp.~262--282, 2000.

\bibitem{abur2004power}
A.~Abur and A.~G. Exposito, {\em Power system state estimation: theory and
  implementation}.
\newblock CRC press, 2004.

\bibitem{huang2012state}
Y.-F. Huang, S.~Werner, J.~Huang, N.~Kashyap, and V.~Gupta, ``State estimation
  in electric power grids: Meeting new challenges presented by the requirements
  of the future grid,'' {\em IEEE Signal Processing Magazine}, vol.~29, no.~5,
  pp.~33--43, 2012.

\bibitem{kekatos2017psse}
V.~Kekatos, G.~Wang, H.~Zhu, and G.~B. Giannakis, {\em PSSE REDUX}, ch.~6,
  pp.~171--208.
\newblock John Wiley \& Sons, Ltd, 2020.

\bibitem{10007691}
S.~Meliopoulos, G.~J. Cokkinides, P.~Myrda, E.~Farantatos, R.~Elmoudi,
  B.~Fardanesh, G.~Stefopoulos, C.~Black, and P.~Panciatici, ``Dynamic
  estimation-based protection and hidden failure detection and identification:
  Inverter-dominated power systems,'' {\em IEEE Power and Energy Magazine},
  vol.~21, no.~1, pp.~59--72, 2023.

\bibitem{Khalili2022}
R.~Khalili, {\em Efficient State and Parameter Estimation in Three-Phase Power
  Systems}.
\newblock PhD thesis, Northeastern University, 2022.

\bibitem{cheng2023survey}
G.~Cheng, Y.~Lin, A.~Abur, A.~G{\'o}mez-Exp{\'o}sito, and W.~Wu, ``A survey of
  power system state estimation using multiple data sources: Pmus, scada, ami,
  and beyond,'' {\em IEEE Transactions on Smart Grid}, vol.~15, no.~1,
  pp.~1129--1151, 2023.

\bibitem{Ghaddar14}
B.~Ghaddar, J.~Marecek, and M.~Mevissen, ``Optimal power flow as a polynomial
  optimization problem,'' {\em IEEE Transactions on Power Systems}, vol.~31, 04
  2014.

\bibitem{molzahn2017survey}
D.~K. Molzahn, F.~D{\"o}rfler, H.~Sandberg, S.~H. Low, S.~Chakrabarti,
  R.~Baldick, and J.~Lavaei, ``A survey of distributed optimization and control
  algorithms for electric power systems,'' {\em IEEE Transactions on Smart
  Grid}, vol.~8, no.~6, pp.~2941--2962, 2017.

\bibitem{cain2012history}
M.~B. Cain, R.~P. O’neill, A.~Castillo, {\em et~al.}, ``History of optimal
  power flow and formulations,'' {\em Federal Energy Regulatory Commission},
  vol.~1, pp.~1--36, 2012.

\bibitem{8855013}
M.~Jin, I.~Molybog, R.~Mohammadi-Ghazi, and J.~Lavaei, ``Scalable and robust
  state estimation from abundant but untrusted data,'' {\em IEEE Transactions
  on Smart Grid}, vol.~11, no.~3, pp.~1880--1894, 2020.

\bibitem{9030243}
M.~Jin, I.~Molybog, R.~Mohammadi-Ghazi, and J.~Lavaei, ``Towards robust and
  scalable power system state estimation,'' in {\em 2019 IEEE 58th Conference
  on Decision and Control (CDC)}, pp.~3245--3252, 2019.

\bibitem{steinmetz1893complex}
C.~P. Steinmetz, ``Complex quantities and their use in electrical
  engineering,'' in {\em Proceedings of the International Electrical Congress},
  pp.~33--74, 1893.

\bibitem{phadke2002synchronized}
A.~G. Phadke, ``Synchronized phasor measurements - a historical overview,'' in
  {\em IEEE/PES transmission and distribution conference and exhibition},
  vol.~1, pp.~476--479, IEEE, 2002.

\bibitem{bale20}
E.~Balestrieri, L.~De~Vito, F.~Picariello, S.~Rapuano, and I.~Tudosa, ``A
  review of accurate phase measurement methods and instruments for sinewave
  signals,'' {\em ACTA IMEKO}, vol.~9, p.~52, 06 2020.

\bibitem{Cheng2023ModelingOS}
G.~Cheng and Y.~Lin, ``Modeling of scada and pmu measurement chains,'' 2023.

\bibitem{mclachlan1988mixture}
G.~J. McLachlan and K.~E. Basford, {\em Mixture models: Inference and
  applications to clustering}, vol.~38.
\newblock M. Dekker New York, 1988.

\bibitem{vsmidl2006variational}
V.~{\v{S}}m{\'\i}dl and A.~Quinn, {\em The variational Bayes method in signal
  processing}.
\newblock Springer Science \& Business Media, 2006.

\bibitem{matpower}
R.~D. Zimmerman, C.~E. Murillo-Sánchez, and R.~J. Thomas, ``Matpower:
  Steady-state operations, planning, and analysis tools for power systems
  research and education,'' {\em IEEE Transactions on Power Systems}, vol.~26,
  no.~1, pp.~12--19, 2011.

\bibitem{zhang2017conic}
Y.~Zhang, R.~Madani, and J.~Lavaei, ``Conic relaxations for power system state
  estimation with line measurements,'' {\em IEEE Transactions on Control of
  Network Systems}, vol.~5, no.~3, pp.~1193--1205, 2017.

\bibitem{huber2004robust}
P.~J. Huber, {\em Robust statistics}, vol.~523.
\newblock John Wiley \& Sons, 2004.

\bibitem{wie2013}
W.~Wiesemann, A.~Tsoukalas, P.~M. Kleniati, and B.~Rustem, ``Pessimistic
  bilevel optimization,'' {\em SIAM Journal on Optimization}, vol.~23, 02 2013.

\bibitem{Chuong23}
T.~Chuong, X.~Yu, A.~Eberhard, C.~Li, and C.~Liu, ``Convergences for robust
  bilevel polynomial programmes with applications,'' {\em Optimization Methods
  and Software}, vol.~38, pp.~1--34, 05 2023.

\bibitem{bertsimas2015best}
D.~Bertsimas, A.~King, and R.~Mazumder, ``{Best subset selection via a modern
  optimization lens},'' {\em The Annals of Statistics}, vol.~44, no.~2, pp.~813
  -- 852, 2016.

\bibitem{green04}
E.~Greenshtein and Y.~Ritov, ``Persistence in high-dimensional linear
  predictor-selection and the virtue of over-parametrization,'' {\em
  Bernoulli}, vol.~10, 12 2004.

\bibitem{dono06}
D.~Donoho, ``For most large underdetermined systems of linear equations the
  minimal l(1)-norm solution is also the sparsest solution.,'' {\em
  Communications on Pure and Applied Mathematics}, vol.~59, pp.~797 -- 829, 06
  2006.

\bibitem{zhao06}
P.~Zhao and B.~Yu, ``On model selection consistency of lasso,'' {\em Journal of
  Machine Learning Research}, vol.~7, pp.~2541--2563, 12 2006.

\bibitem{Marecek2014PowerFA}
J.~Marecek, T.~M. McCoy, and M.~Mevissen, ``Power flow as an algebraic
  system,'' {\em ArXiv}, vol.~abs/1412.8054, 2014.

\bibitem{Baillieul82}
J.~Baillieul and C.~I. Byrnes, ``Remarks on the number of solutions to the load
  flow equations for a power system with electrical losses,'' in {\em 1982 21st
  IEEE Conference on Decision and Control}, pp.~919--924, 1982.

\bibitem{zhang2018spurious}
R.~Y. Zhang, J.~Lavaei, and R.~Baldick, ``Spurious critical points in power
  system state estimation,'' in {\em Hawaii International Conference on System
  Sciences 2018}, 2018.

\bibitem{lesieutre2005convexity}
B.~C. Lesieutre and I.~A. Hiskens, ``Convexity of the set of feasible
  injections and revenue adequacy in ftr markets,'' {\em IEEE Transactions on
  Power Systems}, vol.~20, no.~4, pp.~1790--1798, 2005.

\bibitem{lesieutre2011examining}
B.~C. Lesieutre, D.~K. Molzahn, A.~R. Borden, and C.~L. DeMarco, ``Examining
  the limits of the application of semidefinite programming to power flow
  problems,'' in {\em 2011 49th annual Allerton conference on communication,
  control, and computing (Allerton)}, pp.~1492--1499, IEEE, 2011.

\bibitem{bukhsh2013local}
W.~A. Bukhsh, A.~Grothey, K.~I. McKinnon, and P.~A. Trodden, ``Local solutions
  of the optimal power flow problem,'' {\em IEEE Transactions on Power
  Systems}, vol.~28, no.~4, pp.~4780--4788, 2013.

\bibitem{molzahn2016convex}
D.~K. Molzahn and I.~A. Hiskens, ``Convex relaxations of optimal power flow
  problems: An illustrative example,'' {\em IEEE Transactions on Circuits and
  Systems I: Regular Papers}, vol.~63, no.~5, pp.~650--660, 2016.

\bibitem{molzahn2016visualizing}
D.~Molzahn, ``Visualizing the feasible spaces of challenging opf problems,''
  2016.
\newblock FERC Staff Technical Conference on Increasing Real-Time and Day-Ahead
  Market Efficiency through Improved Software June 28, 2016.

\bibitem{molzahn2019survey}
D.~K. Molzahn, I.~A. Hiskens, {\em et~al.}, ``A survey of relaxations and
  approximations of the power flow equations,'' {\em Foundations and
  Trends{\textregistered} in Electric Energy Systems}, vol.~4, no.~1-2,
  pp.~1--221, 2019.

\bibitem{zhu1997algorithm}
C.~Zhu, R.~H. Byrd, P.~Lu, and J.~Nocedal, ``Algorithm 778: L-bfgs-b: Fortran
  subroutines for large-scale bound-constrained optimization,'' {\em ACM
  Transactions on mathematical software (TOMS)}, vol.~23, no.~4, pp.~550--560,
  1997.

\bibitem{2020SciPy-NMeth}
P.~Virtanen, R.~Gommers, T.~E. Oliphant, M.~Haberland, T.~Reddy, D.~Cournapeau,
  E.~Burovski, P.~Peterson, W.~Weckesser, J.~Bright, S.~J. {van der Walt},
  M.~Brett, J.~Wilson, K.~J. Millman, N.~Mayorov, A.~R.~J. Nelson, E.~Jones,
  R.~Kern, E.~Larson, C.~J. Carey, {\.I}.~Polat, Y.~Feng, E.~W. Moore,
  J.~{VanderPlas}, D.~Laxalde, J.~Perktold, R.~Cimrman, I.~Henriksen, E.~A.
  Quintero, C.~R. Harris, A.~M. Archibald, A.~H. Ribeiro, F.~Pedregosa, P.~{van
  Mulbregt}, and {SciPy 1.0 Contributors}, ``{{SciPy} 1.0: Fundamental
  Algorithms for Scientific Computing in Python},'' {\em Nature Methods},
  vol.~17, pp.~261--272, 2020.

\bibitem{anjos2011handbook}
M.~F. Anjos and J.~B. Lasserre, {\em Handbook on semidefinite, conic and
  polynomial optimization}, vol.~166.
\newblock Springer Science \& Business Media, 2011.

\bibitem{liu2017}
J.~Liu, J.~Marecek, A.~Simonetto, and M.~Takáč, ``A coordinate-descent
  algorithm for tracking solutions in time-varying optimal power flows,'' 10
  2017.

\bibitem{Bellon_2024}
A.~Bellon, D.~Henrion, V.~Kungurtsev, and J.~Mareček, ``Parametric
  semidefinite programming: Geometry of the trajectory of solutions,'' {\em
  Mathematics of Operations Research}, Mar. 2024.

\bibitem{bellon2024time}
A.~Bellon, M.~Dressler, V.~Kungurtsev, J.~Mare{\v{c}}ek, and A.~Uschmajew,
  ``Time-varying semidefinite programming: Path following a burer--monteiro
  factorization,'' {\em SIAM Journal on Optimization}, vol.~34, no.~1,
  pp.~1--26, 2024.

\bibitem{Lessere04}
J.-B. Lasserre, ``Global optimization with polynomials and the problem of
  moments,'' {\em SIAM Journal on Optimization}, vol.~11, 09 2004.

\bibitem{liu2023optimal}
J.~Liu, A.~Bellon, A.~Simonetto, M.~Takac, and J.~Marecek, ``Optimal power flow
  pursuit in the alternating current model,'' 2023.

\bibitem{Lavaei12}
J.~Lavaei and S.~H. Low, ``Zero duality gap in optimal power flow problem,''
  {\em IEEE Transactions on Power Systems}, vol.~27, no.~1, pp.~92--107, 2012.

\bibitem{best19}
K.~Bestuzheva and H.~Hijazi, ``Invex optimization revisited,'' {\em Journal of
  Global Optimization}, vol.~74, 08 2019.

\bibitem{kungurtsev2020twostep}
V.~Kungurtsev and J.~Marecek, ``A two-step pre-processing for semidefinite
  programming,'' 2020.

\bibitem{aminikhanghahi2017survey}
S.~Aminikhanghahi and D.~J. Cook, ``A survey of methods for time series change
  point detection,'' {\em Knowledge and information systems}, vol.~51, no.~2,
  pp.~339--367, 2017.

\bibitem{Gondzio1998WarmSO}
J.~Gondzio, ``Warm start of the primal-dual method applied in the cutting-plane
  scheme,'' {\em Mathematical Programming}, vol.~83, pp.~125--143, 1998.

\bibitem{Polyak63}
B.~Polyak, ``Gradient methods for the minimisation of functionals,'' {\em Ussr
  Computational Mathematics and Mathematical Physics}, vol.~3, pp.~864--878, 12
  1963.

\bibitem{Bolte20}
J.~Bolte and E.~Pauwels, ``Curiosities and counterexamples in smooth convex
  optimization,'' {\em Mathematical Programming}, vol.~195, no.~1,
  pp.~553--603, 2022.

\end{thebibliography}

%\appendix

\begin{IEEEbiographynophoto}
{Pavel Rytir} received his Ph.D. degree in computer science from Charles University in Prague in 2013. He has been a senior researcher at the Czech Technical University since 2017.
\end{IEEEbiographynophoto}
\begin{IEEEbiographynophoto}
{Ales Wodecki} received his Ph.D. degree in mathematical engineering from the Czech Technical University in Prague in 2023. He has been a senior researcher at the Czech Technical University since 2023.
\end{IEEEbiographynophoto}
\begin{IEEEbiographynophoto}
{Martin Malachov} received his Ph.D. degree in mathematical physics from the Czech Technical University in Prague in 2023. He has been a senior specialist at the Czech Transmission System Operator since 2023.
\end{IEEEbiographynophoto}
\begin{IEEEbiographynophoto}
{Pavel Baxant} received his M.Sc. equivalent in mathematical physics from the Czech Technical University in Prague in 2016. He has been at 
the Czech Transmission System Operator
since 2017, currently as a senior specialist.
\end{IEEEbiographynophoto}
\begin{IEEEbiographynophoto}
{Premysl Vorac} (Member, IEEE) received his Ph.D. degree in electrical engineering from the University of West Bohemia in Pilsen in 2018. He has been a technical staff member and manager at 
the Czech Transmission System Operator 
since 2015.
\end{IEEEbiographynophoto}
\begin{IEEEbiographynophoto}
{Miloslava Chladova} (Member, IEEE) received her Ph.D. equivalent in power systems engineering from the Czech Technical University in Prague in 1993. She has been a technical staff member at 
the Czech Transmission System Operator since 2006.
\end{IEEEbiographynophoto}
\begin{IEEEbiographynophoto}
{Jakub Marecek} (Member, IEEE) received a Ph.D. degree in computer science from the University of Nottingham, Nottingham, U.K., in 2012. He has been a faculty member at the Czech Technical University in Prague, the Czech Republic, since 2020. Previously, he had also worked in two start-ups, at ARM Ltd., at the University of Edinburgh, at the University of Toronto, at IBM Research – Ireland, and at the University of California, Los Angeles. His research interests include the design and analysis of algorithms for optimisation and control problems across a range of application domains.
\end{IEEEbiographynophoto}

\vfill

\clearpage 
\appendix 

\subsection{Key Concepts in the Trajectory Tracking of Time-varying SDPs}\label{sec_trajectory_tracking_SDP_genral}
Recently, the trajectories of time-varying semi-definite programs (SDPs) have been charaterized \cite{Bellon_2024}. More specifically, when one considers the primal 
\begin{align}
\begin{split}
\underset{X\in\mathbb{S}^{n}}{\min}\left\langle C\left(t\right),X\right\rangle \\
\text{s.t. }\mathcal{A}\left(t\right)\left[X\right]=b\left(t\right) \\
X\succeq0,\label{eq_time_dep_sdp}
\end{split}
\end{align}
where $\mathbb{S}^n$ denotes the set of symmetric $n$ by $n$ matrices, $b\left(t\right) \in \mathbb{R}^m$ for each $t$, $\left\langle \cdot,\cdot\right\rangle $ is an inner product on the space of matrices and $\mathcal{A}\left(t\right)\left[X\right]=\left(\left\langle A_{1},X\right\rangle ,\left\langle A_{2},X\right\rangle ,\ldots,\left\langle A_{m},X\right\rangle \right)$. 
In order to meaningfully describe the trajectories of the time dependant SDP \eqref{eq_time_dep_sdp} one makes the following definitions
\begin{align}
\begin{split}
\mathcal{P}\left(t\right)=\left\{ X\in\mathbb{S}^{n}:\mathcal{A}\left(t\right)\left[X\right]=b\left(t\right),X\succeq0\right\} , \\
p^{*}\left(t\right)=\underset{X\in\mathbb{S}^{n}}{\min}\left\{ \left\langle C\left(t\right),X\right\rangle :X\in\mathcal{P}\left(t\right)\right\}  \\
\mathcal{P}^{*}\left(t\right)=\left\{ X\in\mathcal{P}\left(t\right):\left\langle C\left(t\right),X\right\rangle =p^{*}\left(t\right)\right\}, 
\end{split}    
\end{align}
where $\mathcal{P}^{*}\left(t\right)$ represents the set-valued map that represents all the solutions of \eqref{eq_time_dep_sdp} at time $t$. 

The usual definition of continuity of real-valued maps does not directly extend to set-valued maps. For this reason, Painleve-Kuratowski continuity is considered. To properly define the notion, we introduce inner and outer limits.

\begin{definition}[inner and outer limits]\label{def_inner_outer_limits}
Let $F$ be a set valued function from $T \subset \mathbb{R}$ to the set of $n$ by $n$ matrices denoted by $X$. Then the inner and outer limits at point $t_0 \in T$ is defined as
\begin{align}
\begin{split}
\underset{t\rightarrow t_{0}}{\liminf}F\left(t\right)=\left\{ x_{0}:\forall\left\{ t_{k}\right\} _{k=1}^{\infty}\subset T\text{ such that }t_{k}\rightarrow t_{0},\right. \\  \left.\exists\left\{ x_{k}\right\} _{k=1}^{\infty}\subset X,x_{k}\rightarrow x_{0},x_{k}\in F\left(t_{k}\right)\right\}, \\  
\underset{t\rightarrow t_{0}}{\liminf}F\left(t\right)=\left\{ x_{0}:\exists\left\{ t_{k}\right\} _{k=1}^{\infty}\subset T\text{ such that }t_{k}\rightarrow t_{0},\right. \\
\left.\exists\left\{ x_{k}\right\} _{k=1}^{\infty}\subset X,x_{k}\rightarrow x_{0},x_{k}\in F\left(t_{k}\right)\right\} 
\end{split}    
\end{align}
respectively.
\end{definition}

\begin{definition}[continuity of a set valued map]\label{def_continuity}
Let $F$ be a set valued map as in Definition \ref{def_inner_outer_limits}. We call $F$ inner, outer semi-continuous at $t_0$ if
\begin{align}
\begin{split}
\underset{t\rightarrow t_{0}}{\liminf}F\left(t\right)\supset F\left(t_{0}\right), \\
\underset{t\rightarrow t_{0}}{\liminf}F\left(t\right)\subset F\left(t_{0}\right)
\end{split}    
\end{align}
hold, respectively. $F$ is called continuous (Painleve-Kuratowski continuous) at $t_0$ if it is both inner and outer semi-continuous at $t_0$.
\end{definition}
Note that if $F$ happens to be single valued, Definition \ref{def_continuity} coincides with the usual definition of continuity. Additionally, we call a set valued map differentiable at point $t_0$ if it is single valued on some neighborhood of $t_0$ and has a derivative at $t_0$ in the ordinary sense. Lastly, we introduce the basic regularity conditions.
\begin{definition}[basic regularity conditions]\label{def_basic_reg_conds}
We say that a SDP of the form  \eqref{eq_time_dep_sdp} satisfies the basic regularity conditions if
\begin{itemize}
    \item the interior of the feasible set is non-empty i.e. if there exists a strictly feasible point (Slater's condition),
    \item if the data vectors and matrices dependent on time $C\left(t\right),\mathcal{A}\left(t\right)$ and $b\left(t\right)$ are continuous with respect to any given matrix or vector norm on the codomain (continuity is independent of norm choice),
    \item the constraints satisfy the so called linear independence constraint qualification (LICQ), which is satisfied if the matrices $A_{i}$ (forming the operator $\mathcal{A}$) are linearly independent, which makes $\mathcal{A}$ surjective.
\end{itemize}
\end{definition}

With these preliminaries in place, we may introduce a relevant classification of trajectory points of a time-varying SDP in order to be able to link the behavior of the SDP in subsequent time instances \cite{Bellon_2024}.
\begin{theorem}[the classification of trajectories of time-varying SDPs]\label{thm_the_classification_sdp}
Assume that the time-varying SDP satisfies the basic regularity conditions as per Definition \ref{def_basic_reg_conds}. Furthermore, assume that the data are polynomial functions of time. Then the trajectory is comprised only of points that may be characterized as follows.
\begin{itemize}
    \item (regular points) The trajectories are single valued on some neighborhood of the point and differentiable at the point.
    \item (non-differentiable points) The trajectories are single valued on some neighborhood of the point and non-differentiable at the point.
    \item (discontinuous isolated multi-valued points) The trajectories are single valued on some neighborhood of the point excluding the point itself and multivalued at the point.
\end{itemize}
\end{theorem}
It is important to note that the aforementioned theorem guarantees that the following pathological behaviors do not occur \cite{Bellon_2024}.
\begin{itemize}
    \item Any neighborhood of a point may contain many points at which multiple solutions to the problem exist (non-isolated multi values).
    \item A trajectory may split into multiple trajectories at a point while preserving continuity (continuous bifurcation).
    \item Non-continuous split into multiple trajectories (irregular accumulation point).
\end{itemize}
Each of the pathological behaviors listed above would introduce major issues to any tracking algorithm that uses a solution at a previous time to inform the solution at current time. Therefore, it is purposeful to discuss the conditions under which these conditions are satisfied for the problem of state estimation in order to make full use of this classification.

\end{document}